\documentclass[11pt]{article}

\usepackage[utf8]{inputenc}
\usepackage[T1]{fontenc}

\usepackage{microtype} 

\usepackage{amsmath, amsfonts, amssymb, amsthm}
\usepackage{mathtools} 
\usepackage{bm}

\usepackage{booktabs} 
\usepackage{enumitem}
\usepackage{thmtools}  

\usepackage[
    style=alphabetic,
    backend=biber,
    maxbibnames=99,maxalphanames=3,
minalphanames=3
]{biblatex}
\AtBeginBibliography{\small\setlength{\emergencystretch}{1em}}
\usepackage[colorlinks=true, allcolors=blue]{hyperref}

\theoremstyle{definition} 
\newtheorem{theorem}{Theorem}[section]
\newtheorem{lemma}[theorem]{Lemma}
\newtheorem{proposition}[theorem]{Proposition}
\newtheorem{corollary}[theorem]{Corollary}

\newtheorem{definition}[theorem]{Definition}

\theoremstyle{remark}
\newtheorem{remark}[theorem]{Remark}
\newtheorem{example}[theorem]{Example}

\newlist{romannum}{enumerate}{1}
\setlist[romannum]{
    label=\textup{(\roman*)},
    ref=\textup{(\roman*)},
    leftmargin=*,
    topsep=5pt,
    itemsep=2pt
}

\usepackage{xcolor}

\newcommand{\Zgeq}{\mathbb{Z}_{\geq 0}}       
\newcommand{\Zgeqn}{\mathbb{Z}_{\geq 0}^n}     
\newcommand{\dnr}{\Delta_n^r}
\newcommand{\delplus}{\delta_J^+}
\newcommand{\delminus}{\delta_J^-}
\newcommand{\ones}{\textbf{1}}
\newcommand{\R}{\mathbb{R}}
\DeclareMathOperator{\supp}{supp}

\newcommand{\hess}{\mathcal{H}}
\newcommand{\Lnr}{L_n^r}

\newcommand{\PMat}{\mathbf{PMat}}      
\newcommand{\MConv}{\mathbf{MConv}}    
\newcommand{\PMatK}{\mathbf{PMat}_\kappa}
\newcommand{\MConvK}{\mathbf{MConv}_\kappa}

\DeclareMathOperator{\proj}{proj}              
\newcommand{\kproj}{\proj_{\kappa}}            

\DeclareMathOperator{\Ret}{Ret}                
\newcommand{\RetK}{\Ret_{\kappa}}              

\newcommand{\Jmax}{J^{\max_{\kappa}}} 

\hypersetup{
    pdftitle={Caged Retractions of Polymatroids},
    pdfauthor={Ari Pomeranz}
}

\usepackage[margin=1.5in]{geometry}

\title{Caged Retractions of Polymatroids}
\author{Ari Pomeranz}
\date{August 2026}

\begin{document}

\maketitle
\begin{abstract}
We develop a unified theory of caged retractions of discrete polymatroids. Given a polymatroid and a cage $\kappa$, the $\kappa$-retraction is a canonical $\kappa$-caged polymatroid obtained by projecting bases into the cage and retaining the maximal projected bases. We prove that this construction agrees with an explicit rank-function formula. We show that the inclusion of the $\kappa$-caged polymatroids into all
polymatroids and the $\kappa$-retraction form a Galois connection with
respect to the weak-map order. As applications, we obtain caged versions of polymatroid union, the disjoint basis theorem, and induction along a bipartite graph. When $\kappa=\ones$, these recover the corresponding matroid constructions. We also study how caged retractions interact with Lorentzian polynomials and representations over near-idempotent tracts. In each case, the construction preserves the relevant structure.
\end{abstract}

\section{Introduction}

\textit{Polymatroids} are combinatorial objects which generalize matroids by allowing bases to be multisets. They have recently played an important role in combinatorial Hodge theory, particularly through the theory of Lorentzian polynomials \cite{branden2024lorentzianpolynomials}. This motivates the study of natural constructions and operations on polymatroids.

The focus of this paper is an operation which takes an arbitrary polymatroid and produces a \emph{caged polymatroid}. When the cage is the unit cage $[0,1]^n$, this operation produces a matroid. In this case, it agrees with the well-known construction discussed in Chapter~11 of \cite{oxleyMatroidTheory2011}, where it is called the \emph{induced matroid} of a polymatroid.

We call this operation the \emph{retraction} of a polymatroid. The retraction can be described both in terms of bases and in terms of rank functions. Related constructions have appeared elsewhere. At the level of independence polytopes, this operation is used in \cite{eur2023intersectiontheorypolymatroids} to define a notion of polymatroid union. At the level of rank functions, the same construction appears, for instance, in Chapter~44 of \cite{schrijver2003combinatorialB}.

The equivalence of the basis and rank-function formulations does not seem to be explicitly recorded in the literature. We prove this equivalence using a poset of polymatroids and Galois connections. This argument also gives a natural proof of the Matroid Union Theorem, as well as a generalization to caged polymatroids. We also prove a caged version of the Disjoint Basis Theorem. This result will be used in forthcoming work on a polymatroidal version of the Shannon Switching Game. Finally, we consider induction of polymatroids along bipartite graphs and show that applying the matroid retraction recovers the classical construction for matroids.

In addition to these structural results, we show that retraction preserves two important kinds of additional structure: the Lorentzian property for polynomials and representations of polymatroids over near-idempotent tracts.

The core combinatorial arguments are self-contained. We use several
standard results from discrete polymatroid theory, discrete convex
analysis, Lorentzian polynomial theory, and polymatroid representation
theory, which are cited where needed. For general background on
discrete polymatroids, we recommend
\cite{herzog2003discretepolymatroids}. Some familiarity with matroid
theory is assumed, and matroids are used throughout as motivation and
as a source of examples.

\paragraph*{Acknowledgments.}
The author thanks Matthew Baker for suggesting this topic and for his guidance throughout the project. His suggestions inspired the Galois connection arguments used in this paper. The author also thanks Tong Jin for initial work which helped clarify the Galois connection in the matroid case.

\paragraph*{AI disclosure.}
Generative AI tools, including ChatGPT and Gemini, were used in preparing
this manuscript to assist with editing, exposition, and checking proofs for
possible gaps or unclear steps. The central mathematical ideas and results
were developed by the author; no principal theorem or lemma originated from
an AI tool. All AI-generated suggestions were independently evaluated and
verified by the author, who takes full responsibility for the content and
correctness of the manuscript.

\section{Preliminaries}

\subsection{Basic Definitions of Polymatroids}

We recall the basic definitions of M-convex sets, polymatroids, caged polymatroids, and embedded minors.

\begin{definition}
    Let
    \[
        \dnr=\{\alpha\in\Zgeqn \mid \alpha_1+\cdots+\alpha_n=r\}.
    \]
    A nonempty set $J\subseteq\dnr$ is an \emph{M-convex set} of rank $r$ on $[n]$ if it satisfies the \emph{symmetric exchange axiom}:
    \begin{align*}
        \tag{SA}\label{def:m_convex_set}
        &\text{for every $\alpha,\beta\in J$ and every $i\in[n]$ with $\alpha_i<\beta_i$, there exists}\\
        &\text{$j\in[n]$ with $\alpha_j>\beta_j$ such that $\alpha+e_i-e_j,\ \beta-e_i+e_j\in J$.}
    \end{align*}
\end{definition}

The elements of $J$ are called the \emph{bases} of the M-convex set. Any nonempty subset of $\Zgeqn$ satisfying the symmetric exchange axiom is automatically contained in $\dnr$ for some $r\in\Zgeq$.

A standard result in discrete convex analysis \cite{murotaDiscreteConvexAnalysis2003} says that the symmetric exchange axiom is equivalent to the following \emph{exchange axiom}:
\begin{align*} \tag{EA} \label{def:weak_exchange_axiom}
        & \text{For every } \alpha,\beta \in J, \text{ and every } i \in [n] \text{ with } \alpha_i < \beta_i, \text{ there exists } \\ &j \in [n] \text{ with } \alpha_j > \beta_j \text{ such that } \alpha +e_i - e_j \in J.
\end{align*}
In practice, this version is often easier to verify.

An M-convex set is a multiset version of the set of bases of a matroid. The following definition is an extension of the rank function of a matroid. 
\begin{definition}
    A (discrete) \emph{polymatroid} is a pair $\mathcal{P}=(E,f)$, where $E$ is a finite set and $f\colon 2^E\to\Zgeq$ satisfies:
    \begin{enumerate}
        \item \emph{Normalized:} $f(\varnothing)=0$.

        \item \emph{Monotone:} if $S\subseteq T\subseteq E$, then $f(S)\leq f(T)$.

        \item \emph{Submodular:} if $S,T\subseteq E$, then
        \[
            f(S\cup T)+f(S\cap T)\leq f(S)+f(T).
        \]
    \end{enumerate}
\end{definition}
The function $f$ is called the \emph{rank function}. Since all polymatroids in this paper are discrete, we omit the word ``discrete.'' Every matroid rank function is a polymatroid rank function, with the additional property that $f(S)\leq |S|$ for all $S\subseteq E$.

The following proposition shows polymatroid rank functions are in bijection with M-convex sets, see \cite{herzog2003discretepolymatroids} for details.
\begin{proposition} \label{prop:M_convex_to_rank}
    The map $J \mapsto f_J$ defined by
    \[
    f_J(A) = \max \big\{ \sum_{i\in A} \alpha_i \mid \alpha \in J \big\}
    \]
    induces a bijection between M-convex sets and polymatroid rank functions. If $J$ is an M-convex set, then $f_J$ is a polymatroid rank function. Conversely, given a polymatroid rank function $f$, the set
    \[
    J_f = \big\{ \alpha \in \Zgeqn \mid \sum_{i \in S} \alpha_i \leq f(S)\ \forall S \subseteq [n] \quad \text{and} \quad \sum_{i \in [n]}\alpha_i = f([n]) \big\}
    \]
    is an M-convex set. Furthermore, these maps are inverses.
\end{proposition}

In light of Proposition~\ref{prop:M_convex_to_rank}, we will identify M-convex sets as a cryptomorphism of polymatroids in terms of bases. For our purposes, we will generally use the terminology corresponding to M-convexity, that is we will discuss M-convex sets $J$ in $\Zgeqn$, with the understanding that these are the bases of some polymatroid $\mathcal{P}$ on the ground set $[n]$. If we are interested in the associated polymatroid rank function, we will make this explicit, which is most relevant in Section~\ref{section:galois_connections_and_applications}. 

Let us consider some examples.
\begin{example}\label{example:matroids_are_polymatroids}
    Let $M$ be a matroid on $[n]$, with set of bases $\mathcal{B}(M)$. We associate to $M$ the set
    \[
        J_M \coloneqq \{\ones_B \mid B \in \mathcal{B}(M)\} \subseteq \{0,1\}^n,
    \]
    where $\ones_B$ denotes the indicator vector of $B$. The base exchange axiom for $M$ is exactly the base exchange axiom for $J_M$, and hence $J_M$ is M-convex. In this way, matroids may be viewed as the M-convex sets contained in the unit cube $\{0,1\}^n$. Conversely, any M-convex set contained in $\{0,1\}^n$ arises as the set of indicator vectors of the bases of a matroid.

    Under this identification, the rank function of $J_M$ is the usual matroid rank function of $M$. In particular, matroids are precisely the polymatroids whose associated M-convex sets are caged by $\ones$.
\end{example}
\begin{example}
    For any $n,r \in \Zgeq$, $\dnr$ is an M-convex set. It is analogous to a multiset version of $U_{r,n}$. For $r \leq n$, one can show $J_{U_{r,n}} = \dnr \cap \{ 0,1\}^n$. 
\end{example}

We equip $\Zgeqn$ with the coordinatewise partial order: for $\alpha,\beta\in\Zgeqn$,
\[
    \alpha\leq\beta
    \quad\Longleftrightarrow\quad
    \alpha_i\leq\beta_i \text{ for every } i\in[n].
\]
For $\alpha\in\Zgeqn$ and $S\subseteq[n]$, write
\[
    \alpha(S)=\sum_{i\in S}\alpha_i,
\]
and let $|\alpha|=\alpha([n])$. For $\alpha,\beta\in\Zgeqn$, we write $\alpha\wedge\beta$ for their coordinatewise minimum and $\alpha\vee\beta$ for their coordinatewise maximum.
\begin{definition}\label{def:caged_polymatroids}
    Let $\kappa\in\Zgeqn$. An M-convex set $J$ is \emph{$\kappa$-caged} if $\alpha\leq\kappa$ for every $\alpha\in J$. The vector $\kappa$ is called a \emph{cage}.
\end{definition}
It is sometimes useful to regard a $\kappa$-caged polymatroid as a pair $(J,\kappa)$. The cage is not required to be minimal. By Example~\ref{example:matroids_are_polymatroids}, matroids are precisely the $\ones$-caged polymatroids, where $\ones$ is the vector with all coordinates equal to $1$.

As an immediate consequence of Proposition~\ref{prop:M_convex_to_rank}, we have the following characterization of caged polymatroids in terms of their rank functions.
\begin{corollary} \label{corollary:caged_polymatroids_rank}
    Let $\kappa \in \Zgeqn$. A polymatroid $\mathcal{P}$ on $[n]$ with rank function $f$ is $\kappa$-caged if and only if for every $S \subseteq [n]$, we have $f(S) \leq \kappa(S)$. Furthermore, by submodularity and monotonicity, this is equivalent to checking only singletons: 
    \[
    f(\{i\}) \leq \kappa_i \quad \text{for all } i \in [n].
    \]
\end{corollary}
Applying this to the matroid case recovers the fact that a matroid rank function $f$ satisfies $f(S) \leq |S|$ for every $S \subseteq [n]$.

\subsection{Embedded Minors of Polymatroids}
We now briefly review the theory of embedded minors of polymatroids introduced in \cite{baker2025representationtheorypolymatroids}, and describe the effect of each embedded minor operation on the cage.

Every polymatroid $J$ has a minimal cage, which we denote by $\delplus$. Precisely, $\delplus$ is the vector defined by \[
(\delplus)_i = \max \{\alpha_i \mid \alpha \in J \}.
\]
A polymatroid $J$ is a matroid if and only if $\delplus \leq \ones$. We also define the vector $\delminus$ by \[
(\delminus)_i = \min \{\alpha_i \mid \alpha \in J \}.
\]
We now define embedded minors of polymatroids.

\begin{definition} \label{def:embedded_minor}
Let $J$ be an M-convex set. Let $\mu\in\Zgeqn$ be \emph{effectively independent} in $J$ ($\exists \alpha \in J, \mu +\delminus \leq \alpha$), and let $\nu\in\Zgeqn$ be \emph{effectively coindependent} in $J$ ($\exists \alpha \in J, \alpha \leq \delplus-\nu$). Let $\tau \in \mathbb{Z}^n$. We define:
\begin{itemize}
\item \emph{Contraction:} $J/\mu \coloneqq \{ \alpha - \mu \mid \alpha \in J, \mu + \delminus \leq \alpha\}$.
\item \emph{Deletion:} $J\setminus \nu \coloneqq \{\alpha \in J \mid \alpha \leq \delplus - \nu \}$.
\item \emph{Translation:} For $\tau \geq -\delminus$, $J + \tau \coloneqq \{ \alpha + \tau \mid \alpha \in J \}$.
\end{itemize}

An \emph{embedded minor} of $J$ consists of an M-convex set
\[
    K=(J\setminus\nu)/\mu+\tau
\]
together with the minor embedding
\[
    \iota_K\colon K\longrightarrow J,
    \qquad
    \gamma\longmapsto\gamma+\mu-\tau,
\]
where $\tau\geq-\delta^-_{(J\setminus\nu)/\mu}$ and there exists $\alpha\in J$ such that $\mu+\delta^-_{J\setminus\nu} \leq\alpha\leq\delta_J^+-\nu.$
\end{definition}

Effective independence guarantees that $J/\mu$ is nonempty, while
effective coindependence guarantees that $J\setminus\nu$ is nonempty.
For an embedded minor, the combined inequality ensures that $\mu$ is
effectively independent after deleting $\nu$. The condition on $\tau$
ensures that the final translated set remains in
$\mathbb{Z}_{\geq 0}^n$. The same underlying M-convex set may arise
from different choices of $\mu$, $\nu$, and $\tau$, and hence with
different minor embeddings.

\begin{remark} \label{remark:embedded_minor_cage}
For deletion and translation, we have \[
\delta_{J\setminus \nu}^+ = \delplus - \nu \quad \text{and} \quad \delta_{J + \tau}^+ = \delplus + \tau.
\]
Thus, deletion and translation transform the minimal cage in the expected way. This is not the case for contraction; see the following example.
\end{remark}

\begin{example}
Consider the M-convex set $J = \{(0,2),(1,1)\}$. Then $J/e_1 = \{(0,1)\}$, so that $\delta_{J/e_1}^+ = (0,1)$. However, $\delplus - e_1 = (0,2)$.
\end{example}

In \cite{baker2025representationtheorypolymatroids}, it is shown that $J/\mu$, $J \setminus \nu$, and $J + \tau$ are M-convex. They also show that $J /(\mu_1 + \mu_2) = (J/\mu_1) /\mu_2$ and $J \setminus(\nu_1 + \nu_2) = (J\setminus \nu_1) \setminus \nu_2$. In particular, these identities allow us to consider deletion, contraction, and translation by a single coordinate at a time in order to understand all embedded minors.

Embedded minors have proven useful for studying representations of polymatroids. The remainder of this paper is focused on studying the $\kappa$-retraction of $J$, which we shall see is always an embedded minor of $J$.

\section{Defining the Retraction of a Polymatroid}
In this section, we define the $\kappa$-retraction $\RetK(J)$ of an M-convex set $J$ and prove structural results relating $\RetK(J)$ to $J$. The main result is a structure theorem describing the fibers of the projection map from $J$ onto its retraction. We also show that $\RetK(J)$ can be realized as an embedded minor of $J$. Intuitively, $\RetK(J)$ should be thought of as a closest $\kappa$-caged polymatroid to $J$.

We first define two operations which will be used to define the $\kappa$-retraction. The first is a projection operation which takes an arbitrary vector in $\Zgeqn$ and projects it into the cage. The second is a maximality operation which takes a set of vectors and returns the subset of vectors whose image under the projection is coordinatewise maximal.

Throughout this section, $J$ is an M-convex set on $[n]$ of rank $r$ and we fix a cage $\kappa\in\Zgeqn$.

\begin{definition}
    The \emph{$\kappa$-projection} of a vector $v\in\Zgeqn$ is the vector $\kproj(v)$ defined by
    \[
        \kproj(v)_i=\min(v_i,\kappa_i).
    \]
    Equivalently, $\kproj(v)=v\wedge\kappa$. For a subset $X\subseteq\Zgeqn$, define
    \[
        \kproj(X)=\{\kproj(v)\mid v\in X\}.
    \]
\end{definition}

\begin{definition}
    For a set $X\subseteq\Zgeqn$, define $X^{\max_\kappa}$ to be the set of elements $v\in X$ such that $\kproj(v)$ is coordinatewise maximal in $\kproj(X)$. We say that such an element is \emph{$\kappa$-maximal in $X$}. A set satisfying $X=X^{\max_\kappa}$ is called \emph{$\kappa$-maximal}.
\end{definition}
We may now define the retraction.
\begin{definition}
    We define the \emph{$\kappa$-retraction} of a polymatroid $J$ to be \[
    \RetK(J) \coloneqq \kproj(J)^{\max_\kappa}, \quad \text{or equivalently} \quad 
    \kproj(\Jmax).
    \]
\end{definition}
From the definition, we have $\RetK(J) = \RetK(\RetK(J))$ as well as $\Ret_{\delplus}(J) = J$. We will see several properties the retraction satisfies.

\begin{example}
    Consider the M-convex set
    \[
        J = \{(3,0),(2,1),(1,2)\}.
    \]
    Let $\kappa=(2,1)$. Then
    \[
        \kproj(J)=\{(2,0),(2,1),(1,1)\}.
    \]
    The unique maximal element of $\kproj(J)$ is $(2,1)$, and hence
    \[
        \Ret_\kappa(J)=\{(2,1)\}.
    \]
\end{example}

Let us show that $\RetK(J)$ is a $\kappa$-caged polymatroid.

\begin{theorem} \label{theorem:retraction_is_polymatroid}
    The set $\RetK(J)$ is a $\kappa$-caged polymatroid.
\end{theorem}

\begin{proof}
    Let $x,y \in \RetK(J)$, with $x_i < y_i$. By definition, there exist $\alpha,\beta \in J$ with $x = \kproj(\alpha)$ and $y = \kproj(\beta)$. Since $y_i \leq \kappa_i$, we must have $\alpha_i = x_i < y_i \leq \beta_i$. Thus by the base exchange axiom, there is some $j \neq i$ with $\alpha_j > \beta_j$ and $\alpha + e_i - e_j \in J$. Since $x_i < \kappa_i$, we have
    \[
        \kproj(\alpha + e_i - e_j)
        =
        \begin{cases}
            x + e_i, & \alpha_j > \kappa_j, \\
            x + e_i - e_j, & \alpha_j \leq \kappa_j.
        \end{cases}
    \]
    Because $x$ is maximal in $\kproj(J)$, the first case is impossible. Hence $\alpha_j \leq \kappa_j$, and so $\kproj(\alpha + e_i - e_j)=x+e_i-e_j$. In particular, $\kappa_j \geq \alpha_j > \beta_j$, so $x_j > y_j$. It remains to show that $x + e_i - e_j$ is maximal in $\kproj(J)$.

    Suppose for contradiction that there exists some $z \in \RetK(J)$ with $x+e_i-e_j \leq z$ and $x+e_i-e_j \neq z$. Let $z = \kproj(\gamma)$ with $\gamma \in J$. Because $x$ is maximal, we must have $z_j = x_j - 1$. Hence $\gamma_j = z_j < \alpha_j$. By the base exchange axiom, there is some $k \in [n]$ with $\alpha_k < \gamma_k$ and $\gamma + e_j - e_k \in J$. Since $\gamma_j < \kappa_j$, we have
    \[
        \kproj(\gamma + e_j - e_k)
        =
        \begin{cases}
            z + e_j, & \gamma_k > \kappa_k, \\
            z + e_j - e_k, & \gamma_k \leq \kappa_k.
        \end{cases}
    \]
    By maximality of $z$, the first case is impossible, so $\gamma_k \leq \kappa_k$ and $z+e_j-e_k \in \kproj(J)$. But then $\alpha_k < \gamma_k \leq \kappa_k$, so $x_k=\alpha_k<\gamma_k=z_k$. Thus $z+e_j-e_k$ dominates $x$. Moreover, since $z \geq x+e_i-e_j$ and $z\neq x+e_i-e_j$, we have $|z|>|x|$. Hence $|z+e_j-e_k|>|x|$, so the domination is strict. This contradicts the maximality of $x$.

    Therefore $x+e_i-e_j$ is maximal in $\kproj(J)$, and hence $x+e_i-e_j \in \RetK(J)$. Finally, every element of $\RetK(J)$ is bounded above by $\kappa$ by definition of $\kproj$. Thus $\RetK(J)$ is a $\kappa$-caged polymatroid.
\end{proof}
\begin{remark}
    When $\kappa=\ones$, the retraction produces an ordinary matroid. We call this special case the \emph{matroid retraction}.
\end{remark}
We will usually reserve the letters $x$ and $y$ for elements of $\RetK(J)$, and letters $\alpha$ and $\beta$ for elements of $J$. As a consequence of Theorem~\ref{theorem:retraction_is_polymatroid}, we have the following alternate description of $\Jmax$.
\begin{corollary} \label{corollary:Jmax_size_description}
    We have
    \[
        \Jmax = \{\alpha\in J \mid |\kproj(\alpha)| \text{ is maximal}\}.
    \]
\end{corollary}
\begin{proof}
    Let $\alpha \in J$ be such that $|\kproj(\alpha)|$ is maximal. Then $\kproj(\alpha)$ is maximal with respect to coordinatewise order, so $\kproj(\alpha) \in \RetK(J)$. Because $\RetK(J)$ is a polymatroid, all elements have the same size, so all elements of $\RetK(J)$ have size $|\kproj(\alpha)|$. Because $\RetK(J) = \kproj(\Jmax)$, the result follows.
\end{proof}

Using this corollary, we have the following lemma.
\begin{lemma} \label{lemma:Jmax_is_Mconvex}
    The set $\Jmax$ is an M-convex set on $[n]$.
\end{lemma}
\begin{proof}
    We verify the exchange axiom. Let $\alpha,\beta \in \Jmax$ and $i \in [n]$ such that $\alpha_i < \beta_i$. Since $\Jmax \subseteq J$, the M-convexity of $J$ implies there exists some $j \in [n]$ such that $\alpha_j > \beta_j$ and \[
    \alpha' = \alpha +e_i -e_j \in J \quad \text{and}\quad \beta' = \beta - e_i +e_j \in J.
    \]
    By Corollary~\ref{corollary:Jmax_size_description}, it suffices to show $|\kproj(\alpha')| = |\kproj(\alpha)|$. There are two cases to consider:
    \begin{enumerate}[label=\textit{Case \arabic*:}, leftmargin=*, align=left]
        \item $\alpha_j > \kappa_j$. In this case, we have $\kproj(\alpha)_j = \kproj(\alpha')_j$. Then clearly $|\kproj(\alpha)| \leq |\kproj(\alpha')|$. By maximality of $\alpha$, we have $|\kproj(\alpha)| = |\kproj(\alpha')|$.
        \item $\alpha_j \leq \kappa_j$. Since $\kappa_j \geq \alpha_j > \beta_j$, we also have $\beta_j \leq \kappa_j$. We claim $\beta_i \leq \kappa_i$. Suppose for contradiction that $\beta_i > \kappa_i$. Then we would have $\kproj(\beta')_i = \kproj(\beta)_i$ and $\kproj(\beta')_j > \kproj(\beta)_j$, implying $|\kproj(\beta')| > |\kproj(\beta)|$, which contradicts the maximality of $\beta$.
        
        Therefore, $\beta_i \leq \kappa_i$. Since $\alpha_i < \beta_i$, this forces $\alpha_i < \kappa_i$. Thus the operation taking $\alpha$ to $\alpha'$ decreases $j$ and increases $i$ in the $\kappa$-projection. Hence we have $|\kproj(\alpha')| = |\kproj(\alpha)|$.
    \end{enumerate}
\end{proof}
Let us now view $\kproj$ as a surjection from $\Jmax$ to $\RetK(J)$. The following lemma shows that each fiber, $\kproj^{-1}(x),$ is an M-convex set.
\begin{lemma}\label{lemma:fibers_Mconvex}
    For any $x \in \RetK(J)$, the fiber $\kproj^{-1}(x)$ is an M-convex set.
\end{lemma}
\begin{proof}
    Let $\alpha,\beta \in \kproj^{-1}(x)$ and let $i \in [n]$ with $\alpha_i < \beta_i$. Since $J$ is M-convex, there exists $j \in [n]$ with $\alpha_j > \beta_j$ and $\alpha+e_i -e_j \in J$. It suffices to show \[
    \kproj(\alpha +e_i -e_j) = x.
    \]
    Since $\beta_i > \alpha_i$ and $\alpha,\beta$ have the same $\kappa$-projection, we have $\alpha_i \geq \kappa_i$. Similarly, $\alpha_j > \beta_j$ implies $\kappa_j \leq \beta_j < \alpha_j$. From this, we have $\kproj(\alpha + e_i -e_j) = \kproj(\alpha) = x$, as desired.
\end{proof}

We now try to understand the relationship between the fibers. The following lemma shows that base exchange between two elements of $\RetK(J)$ induces a base exchange between the corresponding fibers.
\begin{lemma} \label{lemma:fiber_exchange_like_retraction}
    Let $x,y \in \RetK(J)$ with $x_i < y_i$. Then for every $\alpha,\beta \in J$ with $\kproj(\alpha) = x$ and $\kproj(\beta)= y$, we have $\alpha_i = x_i$ and $\beta_i = y_i$.

    Moreover, for any $j \in [n]$ with $\alpha_j > \beta_j$ and $\alpha+e_i -e_j \in J$, we have $\alpha_j = x_j$ and $\beta_j = y_j$.
\end{lemma}
\begin{proof}
    By definition, each fiber is contained in $\Jmax$, thus we may restrict to the case where $J = \Jmax$. Because $x_i < y_i \leq \kappa_i$, we have $x_i = \alpha_i$, which implies $\alpha_i < y_i \leq \beta_i$. By base exchange, there exists $j \in [n]$ with $\alpha_j > \beta_j$ such that \[
    \alpha + e_i -e_j \in J \quad \text{and} \quad \beta-e_i +e_j \in J.
    \]
    Since $J$ is $\kappa$-maximal, and because $\alpha+e_i -e_j$ increases $i$ in its $\kappa$-projection, we must be decreasing $j$ in its $\kappa$-projection. Thus $\alpha_j \leq \kappa_j$ and hence $\beta_j < \alpha_j \leq \kappa_j$. Thus, since $\beta -e_i +e_j$ increases $j$ in its $\kappa$-projection, it must be decreasing $i$ in its $\kappa$-projection. Hence $\beta_i \leq \kappa_i$. Since $\alpha_i < \beta_i$, we have $\alpha_i < \kappa_i$. Thus we have $\alpha_i = x_i$ and $\beta_i = y_i$.

    The moreover follows from similar logic.
\end{proof}

\begin{corollary} \label{corollary:fibers_are_translations}
    For any $x, y \in \RetK(J)$, the fibers are translations of each other. Specifically, if $y = x + v$ for some $v \in \mathbb{Z}^n$, then $\kproj^{-1}(y) = \kproj^{-1}(x) + v$.
\end{corollary}
\begin{proof}
    By the connectivity of the basis-exchange graph of a polymatroid, it suffices to consider the case of an elementary exchange $y = x + e_i - e_j$. Let $\alpha \in \kproj^{-1}(x)$ and $\beta \in \kproj^{-1}(y)$. By Lemma~\ref{lemma:fiber_exchange_like_retraction}, we have $\alpha_i = x_i$ and $\beta_i = y_i$. By the base exchange property, there exists $k \in [n]$ such that $\alpha_k > \beta_k$ and $\alpha + e_i - e_k \in J$. Applying Lemma~\ref{lemma:fiber_exchange_like_retraction} again implies $\alpha_k = x_k$ and $\beta_k = y_k$, which forces $k = j$ as $x$ and $y$ only differ at $i$ and $j$.

    It follows that $\alpha + e_i - e_j \in J$. Since $x_i < \kappa_i$ and $x_j \leq \kappa_j$, we have $\kproj(\alpha + e_i - e_j) = x + e_i - e_j = y$, thus $\alpha + e_i - e_j \in \kproj^{-1}(y)$. This shows that $\kproj^{-1}(x) + (e_i - e_j) \subseteq \kproj^{-1}(y)$. By symmetry, the reverse inclusion holds, yielding the desired equality $
        \kproj^{-1}(y) = \kproj^{-1}(x) + e_i - e_j.$
\end{proof}

We synthesize the above results into the following theorem. We borrow some of the terminology from fiber bundles to describe the structure. 

\begin{theorem}[The Structure Theorem for Polymatroid Retractions] \label{theorem:structure_theorem_for_mconvex}
    The set $\Jmax$ admits the following decomposition:

    \begin{enumerate}
        \item [\textup{(1)}] (Fiber Decomposition) The set $\Jmax$ is partitioned by the fibers of $\RetK(J)$ under the projection $\kproj \colon \Jmax \rightarrow \RetK(J)$:
        \[
        \Jmax = \bigsqcup_{x\in \RetK(J)} \kproj^{-1}(x).
        \]

        \item [\textup{(2)}](Fiber Maps) For any $x,y \in \RetK(J)$ with $y = x + v$ for some $v \in \mathbb{Z}^n$, the map $\phi\colon \kproj^{-1}(x) \to \kproj^{-1}(y)$ given by $\phi(\alpha) = \alpha + v$ is a bijection.

        \item [\textup{(3)}](Global Sections) For every element $\alpha \in \Jmax$, define $\tau \in \Delta_n^{r-r^*}$ by $\tau = \alpha - \kproj(\alpha)$, where $r^*$ is the rank of $\RetK(J)$. This vector $\tau$ defines a section of $\RetK(J)$: for every element $y \in \RetK(J)$, the element $y + \tau$ is contained in $\Jmax$ and satisfies $\kproj(y+\tau) = y$. 
    \end{enumerate}
\end{theorem}
\begin{proof}
\leavevmode
\begin{enumerate}
    \item [\textup{(1)}]This follows immediately from the definition of the $\kappa$-projection and the fact that $\kproj(\Jmax) = \RetK(J)$.
    \item [\textup{(2)}]This is simply a restatement of Corollary~\ref{corollary:fibers_are_translations}.
    \item [\textup{(3)}]Let $\alpha \in \Jmax$, let $x = \kproj(\alpha)$, and $\tau = \alpha - x$. Let $y \in \RetK(J)$. Then $y = x + v$ for some $v \in \mathbb{Z}^n$. By (2), we have $\phi(\alpha) = \alpha + v \in \kproj^{-1}(y)$. Lastly, $\alpha + v = x + v + \tau = y+ \tau$, completing the proof.
\end{enumerate} 
\end{proof}
In the remainder of the paper, we frequently consider a section of
\[
    \kproj\colon\Jmax\to\RetK(J)
\]
determined by an element $\alpha\in\Jmax$, via the vector $\tau=\alpha-\kproj(\alpha)$. We call such a vector $\tau$ a \emph{section-defining vector}. The corresponding section is the translated polymatroid $\RetK(J)+\tau$, and it satisfies $\kproj(\RetK(J)+\tau)=\RetK(J)$.

The following lemma is a simple consequence of the structure theorem and will be used later. 
\begin{lemma}\label{lemma:section_vectors_saturate_coordinates}
    Let $\tau$ be a section-defining vector for the projection
    \[
        \kproj\colon \Jmax\to\RetK(J).
    \]
    If $\tau_i>0$, then $x_i=\kappa_i$ for every $x\in\RetK(J)$.
\end{lemma}

\begin{proof}
    Since $\tau$ is a section vector, we have $\kproj(x+\tau)=x$ for every $x\in\RetK(J)$. Suppose $\tau_i>0$. If $x_i<\kappa_i$, then
    \[
        \kproj(x+\tau)_i>x_i,
    \]
    contradicting $\kproj(x+\tau)=x$. Therefore $x_i=\kappa_i$.
\end{proof}

We now show the retraction is an embedded minor of $J$. First, we need the following lemma.

\begin{lemma}\label{lemma:retraction_of_retraction}
    Suppose $\kappa_1,\kappa_2 \in \Zgeqn$ with $\kappa_1 \leq \kappa_2$. Then, \[\Ret_{\kappa_1}(J) = \Ret_{\kappa_1}\big(\Ret_{\kappa_2}(J)\big).\]
\end{lemma}

\begin{proof}
Let $\alpha\in\Zgeqn$ be arbitrary. Since $\kappa_1\leq \kappa_2$, we have
\[
\proj_{\kappa_1}(\alpha)=\proj_{\kappa_1}(\proj_{\kappa_2}(\alpha)).
\] 
Hence
\[
\Ret_{\kappa_1}(\Ret_{\kappa_2}(J))
=
\left(\proj_{\kappa_1}(J^{\max_{\kappa_2}})\right)^{\max_{\kappa_1}}.
\]

We first show that $\Ret_{\kappa_1}(J)\subseteq \Ret_{\kappa_1}(\Ret_{\kappa_2}(J))$. Let $x\in \Ret_{\kappa_1}(J)$, and choose $\alpha\in J^{\max_{\kappa_1}}$ such that $\proj_{\kappa_1}(\alpha)=x$. Choose $\beta \in J^{\max_{\kappa_2}}$ such that $\proj_{\kappa_2}(\alpha) \leq \proj_{\kappa_2}(\beta)$.

Applying $\proj_{\kappa_1}$ gives $x=\proj_{\kappa_1}(\alpha)\leq \proj_{\kappa_1}(\beta)$. Since $x=\proj_{\kappa_1}(\alpha)$ is maximal in $\proj_{\kappa_1}(J)$, this inequality must be equality. Thus $\proj_{\kappa_1}(\beta)=x$. Since $\beta\in J^{\max_{\kappa_2}}$, this shows that $x\in \proj_{\kappa_1}(J^{\max_{\kappa_2}})$. Moreover, $x$ is maximal in $\proj_{\kappa_1}(J)$, so it is maximal in the subset $\proj_{\kappa_1}(J^{\max_{\kappa_2}})$. Therefore 
\[
x\in \left(\proj_{\kappa_1}(J^{\max_{\kappa_2}})\right)^{\max_{\kappa_1}}=\Ret_{\kappa_1}(\Ret_{\kappa_2}(J)).
\]

Conversely, let $x\in \Ret_{\kappa_1}(\Ret_{\kappa_2}(J))$. Then $x\in \proj_{\kappa_1}(J^{\max_{\kappa_2}})\subseteq \proj_{\kappa_1}(J)$. Since $x$ is maximal in $\proj_{\kappa_1}(J^{\max_{\kappa_2}})$, and since the previous paragraph shows that every maximal element of $\proj_{\kappa_1}(J)$ lies in $\proj_{\kappa_1}(J^{\max_{\kappa_2}})$, it follows that $x$ is maximal in $\proj_{\kappa_1}(J)$. Hence $x\in \Ret_{\kappa_1}(J)$.

Thus $\Ret_{\kappa_1}(\Ret_{\kappa_2}(J))=\Ret_{\kappa_1}(J)$.
\end{proof}

Lemma~\ref{lemma:retraction_of_retraction} allows us to construct any $\kappa$-retraction by a sequence of ``elementary'' retractions, where at each step we retract to $\delplus - e_i$ for some choice of $i \in [n]$. By repeating this process, we eventually arrive at the $\kappa$-retraction.

\begin{remark}\label{remark:kappa_smaller_than_delplus}
    If $\Bar{\kappa}\coloneqq\kappa\wedge\delplus$, then
    $\kproj(\alpha)=\proj_{\Bar{\kappa}}(\alpha)$ for every $\alpha\in J$ and
    \[
        \RetK(J)=\Ret_{\Bar{\kappa}}(J).
    \]
    Thus, when studying the retraction of $J$, it suffices to consider
    $\Bar{\kappa}$ rather than $\kappa$.
\end{remark}

For the remainder of this section, we assume that $\kappa\leq\delplus$.

\begin{theorem} \label{theorem:retraction_is_embedded_minor}
    The polymatroid $\RetK(J)$ is an embedded minor of $J$.
\end{theorem}
\begin{proof}
    Using Lemma~\ref{lemma:retraction_of_retraction} and because embedded minors are closed under composition, it suffices to consider the case of $\kappa = \delplus - e_i$ for some fixed $i \in [n]$. Because $\kappa \in \Zgeqn$, we must additionally assume $(\delplus)_i \geq 1$. We then have two cases to consider:
    \begin{enumerate}[label=\textit{Case \arabic*:}, leftmargin=*, align=left]
        \item If $(\delminus)_i < (\delplus)_i$, then an element $\alpha \in J$ is $\kappa$-maximal if and only if $\alpha_i < (\delplus)_i$. From this, it is easy to see that $\RetK(J) = J \setminus e_i$.
        \item If $(\delminus)_i = (\delplus)_i$, then all elements of $J$ are $\kappa$-maximal, thus $\RetK(J) = J - e_i$.
    \end{enumerate}
\end{proof}
At first glance, this approach seems to provide a canonical way to construct $\RetK(J)$ as an embedded minor of $J$. However, the choice of $i$ at each step is not canonical, and different choices may yield different sequences of embedded minor operations. The following example illustrates this phenomenon.

\begin{example}
Consider the polymatroid 
\[
J = \{(3,0), (2,1), (1,2)\}.
\] 
The coordinatewise maximum vector is $\delplus = (3,2)$. Let us calculate $\Ret_{\ones}(J)$ in two different ways. Notice how the intermediate sets differ depending on the path we take from $(3,2)$ to $(1,1)$:

\begin{itemize}
    \item 
    Retracting to $(2,2)$ yields $\{(2,1), (1,2)\}$. A further retraction to $(1,2)$ isolates $\{(1,2)\}$. Finally, retracting to $(1,1)$ yields $\Ret_{\ones}(J) = \{(1,1)\}$.
    
    \item 
    Retracting to $(3,1)$ yields $\{(3,0), (2,1)\}$. A further retraction to $(2,1)$ isolates $\{(2,1)\}$. Finally, retracting to $(1,1)$ yields $\Ret_{\ones}(J) = \{(1,1)\}$.
\end{itemize}

While both paths give the same matroid $\Ret_{\ones}(J)$, the sequences of intermediate polymatroids are distinct. Notably, the two cases correspond to the two global sections of $\Jmax$, which are $\{(1,2)\}$ and $\{(2,1)\}$ respectively.
\end{example}
This motivates us to characterize all embeddings of $\RetK(J)$ into $J$ through this process of elementary retractions. Each embedding through elementary retractions corresponds to a sequence $(e_{i_1},\dots,e_{i_k})$ with \[
\sum_{m=1}^k e_{i_m} = \delplus - \kappa.
\]
In \cite{baker2025representationtheorypolymatroids}, they show $(J + \tau) \setminus \nu = (J \setminus \nu) + \tau$, which is to say we may perform translation after deletions. Thus we may assume the sequence always begins with only deletions and ends with only translations.

Additionally, the order in which we perform the deletions or the translations should not matter. This leads us to the following result.
\begin{theorem} \label{theorem:sections_and_embedded_minors}
    Fix $\kappa \leq \delplus$. Then there is a bijection between the following two sets:
    \begin{itemize}        
        \item The set of embeddings of $\RetK(J)$ into $J$ through elementary retractions starting with deletions and ending with translations, up to reordering of the deletions and translations.
        \item The set of global sections of the map $\kproj:\Jmax \to \RetK(J)$.
    \end{itemize}
    The bijection is given by sending a sequence $(e_{i_1},\dots,e_{i_k})$ with translations given by $(e_{i_l},\dots,e_{i_k})$ to the global section defined by $\tau = \sum_{m=l}^k e_{i_m}$.
\end{theorem}
\begin{proof}
    Let $(e_{i_1}, \dots, e_{i_k})$ be a sequence of elementary retractions with deletions given by $(e_{i_1}, \dots, e_{i_{l-1}})$ and translations given by $(e_{i_l}, \dots, e_{i_k})$. Let $\nu = \sum_{m=1}^{l-1}e_{i_m}$ and $\tau = \sum_{m=l}^k e_{i_m}$ be the total deletion and total translation vectors, respectively, which are invariant up to reordering the deletions or translations. Then by construction we have 
    \[
    \RetK(J) = (J\setminus \nu)-\tau \quad \text{or equivalently} \quad \RetK(J) + \tau = J\setminus \nu,
    \]
    as well as $\RetK(J \setminus \nu) = \RetK(J)$. Thus we have $\RetK(J \setminus \nu) + \tau = J \setminus \nu$. It follows that all elements of $J \setminus \nu$ must be $\kappa$-maximal in $J \setminus \nu$. Because $\RetK(J) = \RetK(J \setminus \nu)$, we see that $J \setminus \nu \subseteq \Jmax$.
    
    Let $\alpha \in J \setminus \nu \subseteq \Jmax$, and let $\tau' = \alpha - \kproj(\alpha)$. By the Structure Theorem, $\RetK(J \setminus \nu) + \tau' \subseteq J \setminus \nu = \RetK(J\setminus \nu) + \tau$. Since these are finite sets of the same cardinality, containment implies equality, which forces $\tau = \tau'$. This shows that $\tau$ defines a global section of $\Jmax$. 

    \smallskip
    Conversely, let $\alpha \in \Jmax$ and let $\tau = \alpha - \kproj(\alpha)$ define a global section of $\Jmax$. Set $\nu = \delplus - \kappa - \tau$, and choose a coordinate step sequence $(e_{i_1}, \dots, e_{i_k})$ such that 
    \[
    \nu = \sum_{m = 1}^{l-1}e_{i_m} \quad \text{and}\quad \tau = \sum_{m = l}^k e_{i_m}.
    \]
    Because $\nu + \tau = \delplus - \kappa$, the sequence $(e_{i_1}, \dots, e_{i_{l-1}}, e_{i_l}, \dots, e_{i_k})$ yields an embedding of $\RetK(J)$ by elementary retractions, which is well-defined up to reordering of the deletions and translations. It suffices to show that the elementary retractions $(e_{i_1}, \dots, e_{i_{l-1}})$ are deletions and $(e_{i_l}, \dots, e_{i_k})$ are translations.

    By an inductive argument where we perform all of $\nu$ first, alongside the characterization in Theorem~\ref{theorem:retraction_is_embedded_minor}, this verification reduces to showing the following two statements:
    \begin{itemize}
        \item If $\nu_i > 0$, then $(\delminus)_i < (\delplus)_i$.
        \item If $\nu = 0$ and $\tau_i > 0$, then $(\delminus)_i = (\delplus)_i$.
    \end{itemize}
    The first statement follows directly, as we have: 
    \[
    \alpha_i = (\kproj(\alpha) + \tau)_i \leq \kappa_i + \tau_i = (\delplus)_i - \nu_i < (\delplus)_i,
    \]
    which implies $\alpha_i < (\delplus)_i$ whenever $\nu_i > 0$.

    For the second statement, suppose $\nu = 0$ and $\tau_i > 0$. Notice that $\alpha_i = \tau_i + \kproj(\alpha)_i$, which implies $\alpha_i > \kproj(\alpha)_i$. This forces $\kproj(\alpha)_i = \kappa_i$. Thus, we have $\alpha_i = \kappa_i + \tau_i = (\delplus)_i$. For contradiction, suppose there exists $\beta \in J$ with $\beta_i < (\delplus)_i$. We apply the base exchange property to find an index $j \in [n]$ with $\beta_j > \alpha_j$ such that $\alpha - e_i + e_j \in J$. This implies: 
    \[
    \kproj(\alpha - e_i + e_j) = \begin{cases}
        \kproj(\alpha), & \alpha_j \geq \kappa_j, \\
        \kproj(\alpha) + e_j, & \alpha_j < \kappa_j.
    \end{cases}
    \]
    Because $\alpha$ is $\kappa$-maximal, we must have $\alpha_j \geq \kappa_j$. But then we again have $\alpha_j = \kappa_j + \tau_j = (\delplus)_j$, contradicting the fact that $\beta_j > \alpha_j$. Thus we must have $\beta_i = \alpha_i = (\delplus)_i$, meaning $(\delminus)_i = (\delplus)_i$.
\end{proof}

As a consequence of this theorem, we have the following description of the retraction as an embedded minor.
\begin{corollary} \label{corollary:description_of_embedded_minor}
    For any $\alpha \in \Jmax$, with $\tau = \alpha - \kproj(\alpha)$ and $\nu = \delplus - \kappa - \tau$, we have \[
    \RetK(J) = J \setminus \nu - \tau.
    \]
\end{corollary}

We close this section with two useful lemmas.

\begin{lemma}\label{lemma:projection_always_independent}
    For any $\alpha\in J$, there exists $\beta\in\Jmax$ such that
    \[
        \kproj(\alpha)\leq\kproj(\beta).
    \]
    Equivalently, $\kproj(\alpha)$ is contained in some basis of $\RetK(J)$.
\end{lemma}
\begin{proof}
    Let $x = \kproj(\alpha)$ and choose $y \in \kproj(J)$ such that $x \leq y$ and $y$ is coordinatewise maximal in $\kproj(J)$. By definition of $\Jmax$ and $\kproj(J)$, there exists $\beta \in \Jmax$ with $y = \kproj(\beta)$. Thus $\kproj(\alpha) \leq \kproj(\beta)$.
\end{proof}

The next lemma is very useful. It shows that the collection of sections $\RetK(J)+\tau$ has the exchange structure of an M-convex set. Consequently, statements about sections can be proved using elementary base exchanges and the connectivity of the basis-exchange graph of a polymatroid.

\begin{lemma} \label{lemma:collection_of_sections_is_M-convex}
    Define the set
    \[
    S = \{\alpha - \kproj(\alpha) \mid \alpha \in \Jmax\}.
    \]
    Thus, $S$ is the set of translation vectors $\tau$ for which $\RetK(J) + \tau$ is a section of $\Jmax$. Then $S$ is an M-convex set.
\end{lemma}
\begin{proof}
    Fix $x \in \RetK(J)$. By Theorem~\ref{theorem:structure_theorem_for_mconvex}, we have
    \[
    S = \{\alpha - x \mid \alpha \in J \text{ and } \kproj(\alpha) = x\}
    = \kproj^{-1}(x) - x.
    \]
    By Lemma~\ref{lemma:fibers_Mconvex}, the set $\kproj^{-1}(x)$ is M-convex. Hence, $S$ is a translation of an M-convex set and is therefore M-convex.
\end{proof}

\section{Galois Connections and Applications}
\label{section:galois_connections_and_applications}
In this section, we formulate caged retraction as a Galois connection and use the resulting rank formula to study caged polymatroid union, disjoint bases, and induction along a bipartite graph.
\subsection{Posets of Polymatroids}

In this subsection, we use weak maps to define a poset on polymatroids. The weak map relation can be expressed either in terms of M-convex sets
or in terms of rank functions. We first show that the cryptomorphism between
these descriptions is an isomorphism of posets. We then show that the
$\kappa$-retraction and the inclusion of $\kappa$-caged polymatroids into all
polymatroids form a Galois connection. The uniqueness of Galois connections
will allow us to identify the rank function of the $\kappa$-retraction.

\begin{definition}\label{def:weak_map_order_mconvex}
    Let $J_1$ and $J_2$ be M-convex sets on $[n]$. We say there is a
    \emph{weak map} from $J_1$ to $J_2$ if, for every $\alpha\in J_2$,
    there exists $\beta\in J_1$ such that $\alpha\leq\beta$. In this case,
    we write
    \[
        J_2\preceq J_1.
    \]
    We call $\preceq$ the \emph{weak-map order} and write $\MConv$ for the
    resulting poset of M-convex sets on $[n]$. For a fixed cage
    $\kappa\in\Zgeqn$, we write $\MConvK$ for the induced subposet of
    $\kappa$-caged M-convex sets.
\end{definition}

\begin{definition}\label{def:weak_map_order_rank}
    Let $f_1$ and $f_2$ be polymatroid rank functions on $[n]$. We say
    there is a \emph{weak map} from $f_1$ to $f_2$ if
    \[
        f_2(A)\leq f_1(A)
        \qquad\text{for every }A\subseteq[n].
    \]
    In this case, we write $f_2\preceq f_1$. We call $\preceq$ the
    \emph{weak-map order} and write $\PMat$ for the resulting poset of
    polymatroid rank functions on $[n]$. For a fixed cage $\kappa$, we write
    $\PMatK$ for the induced subposet of $\kappa$-caged polymatroid rank
    functions.
\end{definition}

The direction of the notation is worth emphasizing. A weak map is written
from the larger object to the smaller object, whereas the order relation is
written with the smaller object first. Thus a weak map $J_1\to J_2$
corresponds to the inequality $J_2\preceq J_1$, and similarly for rank
functions.

We now show that these two relations agree under the cryptomorphism between
M-convex sets and polymatroid rank functions.

\begin{proposition}\label{prop:equivalence_of_posets}
Let $f_1$ and $f_2$ be polymatroid rank functions on $[n]$, and let $J_1$
and $J_2$ be their respective associated M-convex sets via
Proposition~\ref{prop:M_convex_to_rank}. Then the following statements are
equivalent:
\begin{enumerate}
    \item[\textup{(1)}] $f_2\preceq f_1$, or equivalently, there is a weak
    map from $f_1$ to $f_2$.

    \item[\textup{(2)}] For all $\beta \in J_2$ and all
    $A \subseteq [n]$, there exists $\alpha \in J_1$ such that
    \[
        \sum_{i \in A} \beta_i \leq \sum_{i \in A} \alpha_i.
    \]

    \item[\textup{(3)}] For all $\beta \in J_2$ and all
    $A \subseteq [n]$, there exists $\alpha \in J_1$ such that
    \[
        \forall i \in A,\ \beta_i \leq \alpha_i.
    \]

    \item[\textup{(4)}] $J_2\preceq J_1$, or equivalently, there is a weak
    map from $J_1$ to $J_2$.
\end{enumerate}
\end{proposition}

\begin{proof}
\medskip
\noindent{\textup{(1)} $\implies$ \textup{(2)}.}
By Proposition~\ref{prop:M_convex_to_rank}, choose $\alpha\in J_1$ such
that $f_1(A)=\alpha(A)$. Then
\[
    \beta(A)\leq f_2(A)\leq f_1(A)=\alpha(A),
\]
which proves (2).

\medskip
\noindent{\textup{(2)} $\implies$ \textup{(3)}.}
Fix $\beta \in J_2$ and $A \subseteq [n]$. We proceed by induction on
$|A|$. The case $|A| = 0$ is trivial, and the case $|A|=1$ follows immediately from (2).

For any $\alpha \in J_1$, we partition $A$ into three sets:
\[
    A = S_> \cup S_= \cup S_<,
\]
where
\begin{align*}
    S_> &= \{i \in A \mid \alpha_i > \beta_i\}, \\
    S_= &= \{i \in A \mid \alpha_i = \beta_i\}, \\
    S_< &= \{i \in A \mid \alpha_i < \beta_i\}.
\end{align*}
Choose $\alpha \in J_1$ that satisfies
\begin{equation*}\label{ineq:yA}
    \sum_{i \in A} \beta_i \leq \sum_{i \in A} \alpha_i \tag{$\star$}
\end{equation*}
and minimizes the sum
\begin{equation*}\label{ineq:Slower}
    \sum_{i \in S_<} \bigl(\beta_i-\alpha_i\bigr). \tag{$\star\star$}
\end{equation*}
The existence of such an element is guaranteed by (2).

We claim that $S_< = \varnothing$. Assume for contradiction that it is
nonempty. Because of \eqref{ineq:yA}, the set $S_>$ must also be nonempty.
By the induction hypothesis, there exists an element $\alpha' \in J_1$ such
that
\[
    \alpha'_j \geq \beta_j
    \qquad\text{for all }j\in A\setminus S_>.
\]

Take an index $i \in S_<$. Then
$\alpha_i<\beta_i\leq\alpha'_i$, so by the exchange axiom for M-convex
sets, there exists an index $j\in[n]$ with
$\alpha_j>\alpha'_j$ such that
$\alpha+e_i-e_j\in J_1$.

If $j\notin A$, then the vector $\alpha+e_i-e_j$ still satisfies
\eqref{ineq:yA} and strictly decreases \eqref{ineq:Slower}, which
contradicts the minimality of $\alpha$.

If $j\in A$, then since $\alpha'_j\geq\beta_j$ for all elements of
$A\setminus S_>$, we must have $j\in S_>$. Thus
$\alpha+e_i-e_j$ satisfies \eqref{ineq:yA} and again makes
\eqref{ineq:Slower} strictly smaller, a contradiction.

Hence $S_<=\varnothing$, which completes the inductive step and proves
(3).

\medskip
\noindent{\textup{(3)} $\implies$ \textup{(4)}.}
This follows by applying statement (3) to the full ground set $A=[n]$.

\medskip
\noindent{\textup{(4)} $\implies$ \textup{(1)}.}
This follows immediately from the definition of the rank function of a
polymatroid and Proposition~\ref{prop:M_convex_to_rank}.
\end{proof}

\begin{corollary}\label{corr:equivalence_of_posets}
    The cryptomorphism $J\mapsto f_J$ is an order isomorphism
    \[
        \MConv\cong\PMat.
    \]
    It restricts to an order isomorphism
    \[
        \MConvK\cong\PMatK.
    \]
\end{corollary}

\begin{proof}
    By Proposition~\ref{prop:M_convex_to_rank}, the assignment
    $J\mapsto f_J$ is a bijection, and
    Proposition~\ref{prop:equivalence_of_posets} shows that it preserves
    and reflects the weak-map orders. Hence it is an order isomorphism.
    By Corollary~\ref{corollary:caged_polymatroids_rank}, it restricts to the
    stated order isomorphism between the $\kappa$-caged objects.
\end{proof}

\subsection{Caged retractions as a Galois connection}

We first recall the notion of a Galois connection. Let $P$ and $Q$ be
posets. Maps
\[
    L\colon P\longrightarrow Q
    \qquad\text{and}\qquad
    R\colon Q\longrightarrow P
\]
form a \emph{Galois connection}, written $L\dashv R$, if
\[
    L(p)\leq q
    \quad\Longleftrightarrow\quad
    p\leq R(q)
\]
for every $p\in P$ and $q\in Q$. This equivalence implies that both maps
are order-preserving. Importantly, either map in a Galois connection
uniquely determines the other: $R(q)$ is the greatest $p\in P$ satisfying
$L(p)\leq q$, while $L(p)$ is the least $q\in Q$ satisfying
$p\leq R(q)$. We will use this uniqueness to identify the two formulations
of the $\kappa$-retraction.

Throughout this subsection, fix $n\in\Zgeq$ and
$\kappa\in\Zgeqn$. Let
\[
    \iota_\kappa\colon\MConvK\hookrightarrow\MConv
\]
denote the inclusion of posets.

\begin{theorem}\label{thm:reflection_as_M_convex}
    The inclusion $\iota_\kappa$ and the $\kappa$-retraction form a Galois
    connection
    \[
        \iota_\kappa\dashv\RetK.
    \]
    Equivalently, for every M-convex set $J$ on $[n]$, the retraction
    $\RetK(J)$ is the greatest $\kappa$-caged M-convex set $K$ satisfying
    \[
        K\preceq J.
    \]
\end{theorem}

\begin{proof}
    First suppose that $\iota_\kappa(K)\preceq J$. Let $\beta\in K$.
    Because there is a weak map from $J$ to $K$, there exists an element
    $\alpha\in J$ such that $\beta_i\leq\alpha_i$ for every coordinate
    $i\in[n]$. Since $K$ is $\kappa$-caged, we have
    $\beta_i\leq\kappa_i$ for every $i\in[n]$. Combined with our first
    inequality, we have
    \[
        \beta_i
        \leq
        \min(\alpha_i,\kappa_i)
        =
        \kproj(\alpha)_i
        \qquad\text{for all }i\in[n].
    \]
    By Lemma~\ref{lemma:projection_always_independent}, there exists
    $x\in\RetK(J)$ such that $\kproj(\alpha)\leq x$. Therefore,
    $\beta_i\leq x_i$ for all $i\in[n]$, which proves
    $K\preceq\RetK(J)$.

    Conversely, $\kproj(\alpha)\leq\alpha$ for every $\alpha\in J$.
    Therefore, every element $x\in\RetK(J)$ is bounded above by an element
    of $J$, so $\RetK(J)\preceq J$. If
    $K\preceq\RetK(J)$, transitivity gives
    $\iota_\kappa(K)\preceq J$.
\end{proof}

\begin{remark}
    When $\kappa=\ones$, the poset $\MConvK$ consists of matroids on $[n]$
    ordered by weak maps. Thus Theorem~\ref{thm:reflection_as_M_convex}
    says that the matroid retraction is the greatest matroid below the
    original polymatroid in the weak-map order.
\end{remark}

We now define the corresponding $\kappa$-retraction of rank functions and
prove the analogous Galois connection.

\begin{definition}\label{def:rank_function_retraction}
    Let $f\colon 2^{[n]}\to\Zgeq$ be a polymatroid rank function. Define the
    \emph{$\kappa$-retraction} of $f$ to be the function
    $\RetK(f)\colon 2^{[n]}\to\Zgeq$ given by
    \[
        \RetK(f)(X)
        =
        \min_{Y\subseteq X}
        \left(f(Y)+\kappa(X\setminus Y)\right).
    \]
\end{definition}

This formula appears in Chapter 44 of
\cite{schrijver2003combinatorialB}. In Theorem~2.5 of \cite{Lovasz1983},
Lov\'asz shows that the \emph{convolution} of a submodular function and a
modular function is submodular, which includes the retraction as a special
case. For completeness, we include a proof that the $\kappa$-retraction is
a polymatroid rank function.

\begin{theorem}\label{thm:rank_retraction_is_polymatroid}
    Let $f\colon 2^{[n]}\rightarrow\Zgeq$ be a polymatroid rank
    function. Then the $\kappa$-retraction of $f$ is a polymatroid rank
    function and is caged by $\kappa$.
\end{theorem}

\begin{proof}
    Evaluating the definition at a singleton $X=\{i\}$ gives
    $\RetK(f)(\{i\})=\min\{f(\{i\}),\kappa_i\}\leq\kappa_i$.
    Thus, if $\RetK(f)$ is a polymatroid rank function, it is
    $\kappa$-caged by Corollary~\ref{corollary:caged_polymatroids_rank}.

    It remains to verify that $\RetK(f)$ is a polymatroid rank function:

    \begin{itemize}
        \item (Normalized) Minimizing over subsets of $\varnothing$ forces
        the choice $Y=\varnothing$, so
        $\RetK(f)(\varnothing)=f(\varnothing)+\kappa(\varnothing)=0$.

        \item (Monotone) Let $X_1\subseteq X_2$. Let
        $Y_2\subseteq X_2$ be a subset of $X_2$ which minimizes
        $\RetK(f)(X_2)$, and let $Y_1=X_1\cap Y_2$. Notice that
        $Y_1\subseteq Y_2$ and
        \[
            X_1\setminus Y_1
            =
            X_1\setminus Y_2
            \subseteq
            X_2\setminus Y_2.
        \]
        Then by the definition of the minimum and the monotonicity of both
        $f$ and $\kappa$, we have
        \begin{align*}
            \RetK(f)(X_1)
            &\leq f(Y_1)+\kappa(X_1\setminus Y_1) \\
            &\leq f(Y_2)+\kappa(X_2\setminus Y_2) \\
            &=\RetK(f)(X_2).
        \end{align*}

        \item (Submodularity) Let $X_1,X_2\subseteq[n]$. Let
        $Y_1\subseteq X_1$ and $Y_2\subseteq X_2$ be the minimizers for
        $\RetK(f)(X_1)$ and $\RetK(f)(X_2)$, respectively. Choose the test
        sets
        \[
            Y_{\cup}=Y_1\cup Y_2\subseteq X_1\cup X_2
            \qquad\text{and}\qquad
            Y_{\cap}=Y_1\cap Y_2\subseteq X_1\cap X_2.
        \]

        First, by the submodularity of $f$,
        \[
            f(Y_1)+f(Y_2)
            \geq
            f(Y_{\cup})+f(Y_{\cap}).
        \]
        Since $\kappa$ is modular,
        \begin{align*}
            &\kappa(X_1\setminus Y_1)+\kappa(X_2\setminus Y_2)\\
            &\qquad=
            \kappa(X_1)+\kappa(X_2)-\kappa(Y_1)-\kappa(Y_2)\\
            &\qquad=
            \kappa(X_1\cup X_2)+\kappa(X_1\cap X_2)
            -\kappa(Y_\cup)-\kappa(Y_\cap)\\
            &\qquad=
            \kappa((X_1\cup X_2)\setminus Y_\cup)
            +
            \kappa((X_1\cap X_2)\setminus Y_\cap).
        \end{align*}
        Adding these inequalities gives
        \[
            \RetK(f)(X_1)+\RetK(f)(X_2)
            \geq
            \RetK(f)(X_1\cup X_2)+\RetK(f)(X_1\cap X_2),
        \]
        completing the proof.
    \end{itemize}
\end{proof}

Let
\[
    \jmath_\kappa\colon\PMatK\hookrightarrow\PMat
\]
denote the inclusion of posets.

\begin{theorem}\label{thm:rank_retraction_galois}
    The inclusion $\jmath_\kappa$ and the $\kappa$-retraction form a Galois
    connection
    \[
        \jmath_\kappa\dashv\RetK.
    \]
    Equivalently, for every polymatroid rank function $f$ on $[n]$, the
    retraction $\RetK(f)$ is the greatest $\kappa$-caged polymatroid rank
    function $g$ satisfying
    \[
        g\preceq f.
    \]
\end{theorem}

\begin{proof}
    First suppose that $\jmath_\kappa(g)\preceq f$. Fix
    $X\subseteq[n]$. To prove that $g(X)$ is bounded by the retraction, it
    suffices to show that, for every subset $Y\subseteq X$, we have
    \[
        g(X)\leq f(Y)+\kappa(X\setminus Y).
    \]
    For any such $Y$, we have
    \begin{align*}
        g(X)
        &=g(Y\cup(X\setminus Y))\\
        &\leq g(Y)+g(X\setminus Y)\\
        &\leq g(Y)+\kappa(X\setminus Y)\\
        &\leq f(Y)+\kappa(X\setminus Y).
    \end{align*}
    The first inequality follows from the submodularity and normalization
    of $g$. The second holds because $g$ is $\kappa$-caged, and the final
    inequality holds because $\jmath_\kappa(g)\preceq f$. It follows that
    $g\preceq\RetK(f)$.

    Conversely, choosing $Y=X$ in the definition of $\RetK(f)$ gives
    $\RetK(f)(X)\leq f(X)$ for every $X\subseteq[n]$. Thus
    $\RetK(f)\preceq f$. If $g\preceq\RetK(f)$, transitivity gives
    $\jmath_\kappa(g)\preceq f$.
\end{proof}

The order isomorphisms in
Corollary~\ref{corr:equivalence_of_posets} identify the two inclusion
maps. The uniqueness property of Galois connections therefore identifies
the corresponding retraction maps.

\begin{corollary}\label{cor:both_reflections_same}
    Let $f$ be a polymatroid rank function on $[n]$, and let $J$ be its
    associated M-convex set via
    Proposition~\ref{prop:M_convex_to_rank}. Then the rank function of
    $\RetK(J)$ is $\RetK(f)$.
\end{corollary}

\begin{proof}
    Under the order isomorphisms of
    Corollary~\ref{corr:equivalence_of_posets}, the Galois connection of
    Theorem~\ref{thm:reflection_as_M_convex} becomes a Galois connection whose
    left-hand map is
    \[
        \jmath_\kappa\colon\PMatK\hookrightarrow\PMat
    \]
    and whose right-hand map sends $f$ to the rank function of $\RetK(J)$.
    By Theorem~\ref{thm:rank_retraction_galois}, the map
    $f\mapsto\RetK(f)$ is also a right-hand map in a Galois connection with
    $\jmath_\kappa$. Since the right-hand map in a Galois connection is
    uniquely determined by the left-hand map, the two rank functions are
    equal.
\end{proof}

\begin{remark}\label{remark:categorical_retraction}
    When the cage is constant, say $\kappa=m\cdot \ones$, the preceding result admits
    a stronger categorical formulation. The $\kappa$-retraction extends to a
    functor on the category of polymatroids with weak maps, allowing the ground
    set to vary, and this functor is left adjoint to the inclusion of the full
    subcategory of $\kappa$-caged polymatroids.
\end{remark}

\subsection{A Polymatroid Union Theorem}

In this section, we use the $\kappa$-retraction to prove a polymatroid union
theorem, which contains the classical matroid union theorem as a special
case. By using the description of the $\kappa$-retraction as an M-convex set
and rank function, we quickly get the result. Throughout this section, we
fix $n\in\Zgeq$ and $\kappa\in\Zgeqn$.

The proof of the following proposition is adapted from the corresponding
theorem for jump systems presented by Bouchet and Cunningham
\cite{bouchetCunninghamDeltaMatroidsJumpSystems}, which itself credits
Andr\'as Seb\"o.

\begin{proposition}\label{prop:sum_polymatroid_is_polymatroid}
    Let $J_1$ and $J_2$ be M-convex sets on $[n]$. Their Minkowski sum
    \[
        J_1+J_2
        =
        \{\alpha_1+\alpha_2\mid\alpha_i\in J_i\}
    \]
    is M-convex. Furthermore, its corresponding rank function is the sum
    of the rank functions of $J_1$ and $J_2$.
\end{proposition}

\begin{proof}
    We first show that $J_1+J_2$ is M-convex. Let $x,y\in J_1+J_2$, and
    write $x=\alpha_1+\alpha_2$ and $y=\beta_1+\beta_2$, where
    $\alpha_i,\beta_i\in J_i$. Suppose $x_i<y_i$.

    Consider all triples $(\alpha_1',\alpha_2',j)$ such that
    $\alpha_1'\in J_1$, $\alpha_2'\in J_2$, $j\in[n]$, and
    \[
        \alpha_1'+\alpha_2'=x+e_i-e_j.
    \]
    This set is nonempty, since $(\alpha_1,\alpha_2,i)$ is such a triple.
    Choose a triple minimizing
    \[
        D
        =
        \sum_{k\in[n]}
        \left(
            |(\alpha_1')_k-(\beta_1)_k|
            +
            |(\alpha_2')_k-(\beta_2)_k|
        \right).
    \]

    We claim that $x_j>y_j$. Suppose not, so $x_j\leq y_j$. Then
    $(\alpha_1'+\alpha_2')_j<(\beta_1+\beta_2)_j$. Indeed, if $j=i$,
    this follows from $x_i<y_i$, and if $j\neq i$, then
    $(\alpha_1'+\alpha_2')_j=x_j-1<y_j$. Thus either
    $(\alpha_1')_j<(\beta_1)_j$ or
    $(\alpha_2')_j<(\beta_2)_j$. Without loss of generality, assume
    $(\alpha_1')_j<(\beta_1)_j$.

    By the exchange axiom for $J_1$, there exists $k\in[n]$ such that
    $(\alpha_1')_k>(\beta_1)_k$ and
    $\alpha_1'+e_j-e_k\in J_1$. But then
    $(\alpha_1'+e_j-e_k,\alpha_2',k)$ is another valid triple, since
    \[
        (\alpha_1'+e_j-e_k)+\alpha_2'=x+e_i-e_k.
    \]
    Moreover, this replacement moves $\alpha_1'$ closer to $\beta_1$ in
    both the $j$ and $k$ coordinates and does not change any other
    coordinates. Thus the value of $D$ decreases, contradicting the choice
    of our triple.

    Hence $x_j>y_j$. In particular, $j\neq i$, and since
    $\alpha_1'+\alpha_2'=x+e_i-e_j$, this shows that
    $x+e_i-e_j\in J_1+J_2$. Therefore $J_1+J_2$ is M-convex.

    Now let $f_1$, $f_2$, and $f$ be the rank functions of $J_1$, $J_2$,
    and $J_1+J_2$, respectively. For $A\subseteq[n]$, the choices of
    $\alpha_1$ and $\alpha_2$ are independent, so
    \begin{align*}
        f(A)
        &=
        \max\left\{
            \sum_{i\in A}(\alpha_1+\alpha_2)_i
            \mathrel{\Big|}
            \alpha_1\in J_1,\ \alpha_2\in J_2
        \right\}\\
        &=
        \max\left\{
            \sum_{i\in A}(\alpha_1)_i
            \mathrel{\Big|}
            \alpha_1\in J_1
        \right\}
        +
        \max\left\{
            \sum_{i\in A}(\alpha_2)_i
            \mathrel{\Big|}
            \alpha_2\in J_2
        \right\}\\
        &=f_1(A)+f_2(A).
    \end{align*}
    Thus the corresponding rank function is $f_1+f_2$.
\end{proof}

\begin{definition}
    Let $J_1$ and $J_2$ be $\kappa$-caged M-convex sets on $[n]$. We
    define their \emph{$\kappa$-polymatroid union} to be
    \[
        J_1\vee J_2\coloneqq\RetK(J_1+J_2).
    \]
\end{definition}

\begin{theorem}[Polymatroid Union Rank Function]
\label{thm:polymatroid_union_rank}
    Let $\mathcal{P}_1=([n],f_1)$ and
    $\mathcal{P}_2=([n],f_2)$ be polymatroids with associated
    $\kappa$-caged M-convex sets $J_1$ and $J_2$. Then the rank function
    of $J_1\vee J_2$ is
    \[
        f_\vee(X)
        =
        \min_{Y\subseteq X}
        \left(
            f_1(Y)+f_2(Y)+\kappa(X\setminus Y)
        \right).
    \]
\end{theorem}

\begin{proof}
    By Proposition~\ref{prop:sum_polymatroid_is_polymatroid}, the rank
    function of $J_1+J_2$ is $f_1+f_2$. By
    Corollary~\ref{cor:both_reflections_same}, retracting an M-convex set
    and retracting its corresponding rank function give the same
    polymatroid. Applying the rank-function retraction to $f_1+f_2$ gives
    \[
        f_\vee(X)
        =
        \min_{Y\subseteq X}
        \left(
            (f_1+f_2)(Y)+\kappa(X\setminus Y)
        \right),
    \]
    which is the desired formula.
\end{proof}

We now recover the classical matroid union theorem by applying this
construction to the matroid retraction.

\begin{corollary}[Matroid Union Theorem]\label{corr:matroid_union}
    Let $M_1$ and $M_2$ be matroids on a common finite ground set $[n]$,
    with rank functions $r_1$ and $r_2$. The matroid union $M_1\vee M_2$
    has bases
    \[
        \mathcal{B}(M_1\vee M_2)
        =
        \max_{\subseteq}
        \left\{
            B_1\cup B_2
            \mid
            B_1\in\mathcal{B}(M_1),\,
            B_2\in\mathcal{B}(M_2)
        \right\},
    \]
    where $\max_{\subseteq}$ denotes the collection of inclusion-maximal
    members. Its rank function is given by
    \[
        r_{M_1\vee M_2}(X)
        =
        \min_{Y\subseteq X}
        \left(
            r_1(Y)+r_2(Y)+|X\setminus Y|
        \right).
    \]
\end{corollary}

\begin{proof}
    Let $J_1$ and $J_2$ be the M-convex sets of bases of $M_1$ and $M_2$,
    viewed as subsets of $\mathbb{Z}_{\geq0}^n$ through their indicator
    vectors. The matroid retraction is the case $\kappa=\ones$. Thus
    \[
        J_1\vee J_2
        =
        \Ret_{\ones}(J_1+J_2).
    \]
    An element of $J_1+J_2$ has the form
    $\ones_{B_1}+\ones_{B_2}$, with
    $B_i\in\mathcal{B}(M_i)$. Projecting to the $\ones$-cage replaces
    this vector by the indicator vector of $B_1\cup B_2$. Taking maximal
    elements therefore gives precisely the inclusion-maximal sets among
    the unions $B_1\cup B_2$. These are the bases of the usual matroid
    union.

    Finally, applying Theorem~\ref{thm:polymatroid_union_rank} with
    $\kappa=\ones$ gives
    \[
        r_{M_1\vee M_2}(X)
        =
        \min_{Y\subseteq X}
        \left(
            r_1(Y)+r_2(Y)+\ones(X\setminus Y)
        \right).
    \]
    Since $\ones(X\setminus Y)=|X\setminus Y|$, this is exactly the
    classical matroid union rank formula.
\end{proof}
Using Theorem~\ref{thm:polymatroid_union_rank}, we can prove a generalization of the disjoint basis theorem for polymatroids.
\begin{theorem}[Caged Disjoint Basis Theorem]
\label{thm:caged_disjoint_basis}
    Let $J_1,\dots,J_m$ be M-convex sets on $[n]$, with corresponding rank
    functions $r_1,\dots,r_m$. Then there exist bases
    $\alpha_t\in J_t$ such that
    \[
        \alpha_1+\cdots+\alpha_m\leq\kappa
    \]
    if and only if, for every $A\subseteq[n]$,
    \[
        \sum_{t=1}^m
        \left(
            r_t([n])-r_t(A)
        \right)
        \leq
        \kappa([n]\setminus A).
    \]
\end{theorem}

\begin{proof}
    First suppose there exist bases $\alpha_t\in J_t$ with
    $\sum_t\alpha_t\leq\kappa$. For any $A\subseteq[n]$, we have
    $\alpha_t(A)\leq r_t(A)$, and hence
    \[
        \alpha_t([n]\setminus A)
        =
        r_t([n])-\alpha_t(A)
        \geq
        r_t([n])-r_t(A).
    \]
    Summing over $t$ gives
    \[
        \sum_{t=1}^m
        \left(
            r_t([n])-r_t(A)
        \right)
        \leq
        \sum_{t=1}^m\alpha_t([n]\setminus A)
        \leq
        \kappa([n]\setminus A).
    \]

    Conversely, suppose that the inequality holds for every
    $A\subseteq[n]$. Let
    \[
        J=J_1+\cdots+J_m.
    \]
    By iterating
    Proposition~\ref{prop:sum_polymatroid_is_polymatroid}, the rank
    function of $J$ is $r_1+\cdots+r_m$. Therefore the rank of the
    retraction $\RetK(J)$ is
    \[
        \min_{A\subseteq[n]}
        \left(
            \sum_{t=1}^m r_t(A)
            +
            \kappa([n]\setminus A)
        \right).
    \]
    By assumption, each term in this minimum is at least
    $\sum_t r_t([n])$, while equality occurs when $A=[n]$. Thus
    $\RetK(J)$ has the same total rank as $J$.

    Choose a basis $x$ of $\RetK(J)$. By construction,
    $x=\kproj(\alpha)$ for some $\alpha\in J$. Since $x$ and $\alpha$ have
    the same size, the projection did not decrease any coordinate of
    $\alpha$, and hence $x=\alpha$. Thus $x$ is a basis of $J$ which is
    bounded by $\kappa$. Writing
    \[
        x=\alpha_1+\cdots+\alpha_m
    \]
    with $\alpha_t\in J_t$, we obtain bases satisfying
    $\alpha_1+\cdots+\alpha_m\leq\kappa$.
\end{proof}

When the cage is $\ones$, the theorem specializes to the usual matroid
setting. Taking $J_1=\cdots=J_m$ to be the base set of a single matroid
recovers the classical disjoint basis theorem.

\begin{corollary}[Disjoint Basis Theorem]
\label{corr:disjoint_basis_theorem}
    Let $M$ be a matroid on $[n]$ with rank function $r$. Then $M$ has $m$
    pairwise disjoint bases if and only if, for every $A\subseteq[n]$,
    \[
        m\left(r([n])-r(A)\right)
        \leq
        |[n]\setminus A|.
    \]
\end{corollary}

\begin{proof}
    Let
    \[
        J
        =
        \{\ones_B\mid B\text{ is a basis of }M\}
    \]
    be the M-convex set associated to the matroid $M$. Apply
    Theorem~\ref{thm:caged_disjoint_basis} with
    $J_1=\cdots=J_m=J$ and $\kappa=\ones$. A vector inequality
    \[
        \ones_{B_1}+\cdots+\ones_{B_m}\leq\ones
    \]
    is exactly the condition that the bases $B_1,\dots,B_m$ are pairwise
    disjoint. Since
    $\ones([n]\setminus A)=|[n]\setminus A|$, the theorem gives precisely
    the stated condition.
\end{proof}

\subsection{Induction Along a Bipartite Graph}
We now apply caged retraction to induction along a bipartite graph. Polymatroid induction along a bipartite graph goes back to McDiarmid \cite{mcdiarmidRadosTheoremPolymatroids1975} and is also a special case of Schrijver's induction through a poly-linking system \cite[Chapter~6]{schrijverMatroidsLinkingSystems1978}.

Let $G$ be a bipartite graph with left vertex set $[n]$ and right vertex set $[m]$, and let $J\subseteq\mathbb{Z}_{\geq 0}^n$ be an M-convex set with rank function $f$.
We regard $E(G)$ as a subset of $[n]\times[m]$ and write $ij$ for the edge $(i,j)$. For $A\subseteq[m]$, define its \emph{neighborhood} by
\[
    N_G(A)
    \coloneqq
    \{i\in[n]\mid ij\in E(G)\text{ for some }j\in A\}.
\]
We assume that $N_G([m])=[n]$. Otherwise, we first restrict $f$ and $G$ to the left vertex set $N_G([m])$.

For $w\in\mathbb{Z}_{\geq 0}^{E(G)}$, extend $w$ to $[n]\times[m]$ by setting $w_{ij}=0$ whenever $ij\notin E(G)$, and define
\[
    \begin{aligned}
        d_L(w)_i&=\sum_{j\in[m]}w_{ij},
        &&i\in[n],\\
        d_R(w)_j&=\sum_{i\in[n]}w_{ij},
        &&j\in[m].
    \end{aligned}
\]

\begin{definition}
    The \emph{induction of $J$ along $G$} is
    \[
        \operatorname{Ind}_G(J)
        =
        \left\{
            d_R(w):
            w\in\mathbb{Z}_{\geq 0}^{E(G)}
            \text{ and }
            d_L(w)\in J
        \right\}.
    \]
\end{definition}

\begin{proposition}\label{prop:bipartite_induction_rank}
    The set $\operatorname{Ind}_G(J)$ is M-convex, with rank function
    \[
        f_G(A)=f\bigl(N_G(A)\bigr)
        \qquad
        \text{for every }A\subseteq[m].
    \]
\end{proposition}

\begin{proof}
    The construction above is the bipartite specialization of the network induction considered by Murota, so Theorem~3.7(2) of \cite{murota2020basicoperationsrelatednetwork} shows that $\operatorname{Ind}_G(J)$ is M-convex.

    Let $g$ denote its rank function. If $\beta=d_R(w)\in\operatorname{Ind}_G(J)$, then every edge whose weight contributes to $\beta(A)$ has its left endpoint in $N_G(A)$. Therefore,
    \[
        \beta(A)
        \leq d_L(w)\bigl(N_G(A)\bigr)
        \leq f\bigl(N_G(A)\bigr).
    \]
    Taking the maximum over $\beta\in\operatorname{Ind}_G(J)$ gives
    \[
        g(A)\leq f\bigl(N_G(A)\bigr).
    \]

    For the reverse inequality, choose $\alpha\in J$ such that
    $\alpha\bigl(N_G(A)\bigr)=f\bigl(N_G(A)\bigr)$. For each
    $i\in N_G(A)$, choose a neighbor $j_i\in A$, and for each remaining
    $i$, choose any neighbor $j_i$. Then $j_i\in A$ exactly when
    $i\in N_G(A)$. Define $w$ by placing weight $\alpha_i$ on $ij_i$
    for each $i\in[n]$ and weight zero on every other edge. We then have
    $d_L(w)=\alpha$ and
    \[
        d_R(w)(A)
        =
        \sum_{i\in N_G(A)}\alpha_i
        =
        f\bigl(N_G(A)\bigr).
    \]
    Thus $g(A)\geq f\bigl(N_G(A)\bigr)$, proving the formula.
\end{proof}

Suppose now that $M$ is a matroid on $[n]$. In Theorem~11.2.12 of \cite{oxleyMatroidTheory2011}, the matroid induced from $M$ by $G$ is defined by declaring $X\subseteq[m]$ independent when $X$ can be matched in $G$ to an independent set of $M$.

\begin{corollary}
    If $M$ is a matroid on $[n]$, then
    \[
        J_{\operatorname{Ind}_G(M)}
        =
        \Ret_{\ones}\bigl(\operatorname{Ind}_G(J_M)\bigr).
    \]
\end{corollary}

\begin{proof}
    By Proposition~\ref{prop:bipartite_induction_rank} and the rank formula for retraction, the rank function of the right-hand side is
    \[
        A\longmapsto
        \min_{B\subseteq A}
        \left(
            r_M\bigl(N_G(B)\bigr)+|A\setminus B|
        \right).
    \]
    By Corollary~11.2.14 of \cite{oxleyMatroidTheory2011}, this is the rank function of $\operatorname{Ind}_G(M)$.
\end{proof}

The corollary identifies ordinary matroid induction with the matroid
retraction of polymatroid induction. This suggests imposing an
arbitrary coordinatewise bound on the output by applying
$\kappa$-retraction.

\begin{definition}
    Let $\kappa\in\mathbb{Z}_{\geq 0}^m$. The \emph{$\kappa$-caged induction of $J$ along $G$} is
    \[
        \operatorname{Ind}_G^\kappa(J)
        =
        \Ret_\kappa\bigl(\operatorname{Ind}_G(J)\bigr).
    \]
\end{definition}

Equivalently, $\operatorname{Ind}_G^\kappa(J)$ consists of the coordinatewise maximal elements of
\[
    \left\{
        d_R(w)\wedge\kappa:
        w\in\mathbb{Z}_{\geq 0}^{E(G)}
        \text{ and }
        d_L(w)\in J
    \right\}.
\]
Its rank function is
\[
    f_G^\kappa(A)
    =
    \min_{B\subseteq A}
    \left(
        f\bigl(N_G(B)\bigr)
        +
        \kappa(A\setminus B)
    \right).
\]
\begin{remark}
    The preceding construction also contains polymatroid union. Let $J_1$ and $J_2$ be M-convex sets on $[n]$, and consider the M-convex set
    \[
        J_1\oplus J_2
        =
        \{(\alpha_1,\alpha_2)\mid
        \alpha_1\in J_1,\ \alpha_2\in J_2\}
    \]
    on two disjoint copies of $[n]$. Let $G$ have left vertex set
    $[n]\times\{1,2\}$ and right vertex set $[n]$, where $(i,t)$ is
    adjacent only to $i$. Then
    \[
        \operatorname{Ind}_G(J_1\oplus J_2)=J_1+J_2.
    \]
    If, in addition, $J_1$ and $J_2$ are $\kappa$-caged, then
    \[
        \operatorname{Ind}_G^\kappa(J_1\oplus J_2)
        =
        \Ret_\kappa(J_1+J_2)
        =
        J_1\vee J_2.
    \]
    Thus Proposition~\ref{prop:sum_polymatroid_is_polymatroid} and
    Theorem~\ref{thm:polymatroid_union_rank} are special cases of
    Proposition~\ref{prop:bipartite_induction_rank} and its caged version.
\end{remark}

\section{Polynomials Supported on Polymatroids}\label{section:polynomials_supported_on_polymatroids}

The goal of this section is to set up a common language for Lorentzian polynomials and representations of polymatroids. Both may be viewed as assigning coefficients to the points of an M-convex set. From this point of view, the embedded minor operations on M-convex sets induce corresponding operations on the coefficients.

Embedded minors of polymatroid representations were studied in \cite{baker2025representationtheorypolymatroids}. The definitions below are equivalent to the corresponding operations in \cite{baker2025representationtheorypolymatroids}, but are phrased in terms of exponential generating functions supported on M-convex sets. This has the advantage that the same notation applies both to Lorentzian polynomials and to representations over tracts. In particular, Theorem~\ref{theorem:sections_and_embedded_minors} allows us to define retractions of these objects by choosing a section of $\RetK(J)$ inside $J$.

\subsection{Monoids and Tracts}

We begin by recalling the algebraic objects which will serve as the coefficients.

\begin{definition}
    A \emph{monoid} is a set $F$ together with a binary operation $\cdot \colon F\times F\to F$ which is associative and has an identity element $1\in F$. We say that $F$ is \emph{commutative} if $a\cdot b=b\cdot a$ for all $a,b\in F$.

    A monoid $F$ has a \emph{zero} if there is an element $0\in F$ such that $0\cdot a=a\cdot 0=0$ for all $a\in F$.
\end{definition}

For Lorentzian polynomials, the relevant monoid is $\R_{\geq 0}$ under ordinary multiplication. For representations, the monoid carries the additional structure of a \emph{tract}.

Intuitively, a tract is a field-like object in which addition is not necessarily a binary operation. Instead, one keeps track of which formal sums are declared to be zero.

\begin{definition}
    A \emph{tract} is a pair $(F,N_F)$, where $F$ is a commutative monoid with zero such that
    \[
        F^\times = F\setminus \{0\}
    \]
    is an abelian group, and $N_F$ is an ideal in the semiring of finite formal sums
    \[
        \mathbb{N}[F^\times]
        =
        \left\{
            \sum_{i=1}^m n_i a_i
            \ \middle|\
            a_i\in F^\times,\ n_i\in\Zgeq
        \right\}.
    \]
    The ideal $N_F$ is called the \emph{null set}. We require that, for every $a\in F$, there is a unique $b\in F$ such that the formal sum $a+b$ belongs to $N_F$. We call this element $b$ the \emph{additive inverse} of $a$ and denote it by $-a$.

    The monoid $F$ embeds into $\mathbb{N}[F^\times]$ by sending $0$ to the empty sum, which is the additive identity of $\mathbb{N}[F^\times]$, and sending $a\in F^\times$ to the one-term formal sum $a=1\cdot a$.
\end{definition}

\begin{example}
    We list three basic examples of tracts.
    \begin{enumerate}
        \item \textbf{Fields.} Let $k$ be a field. The associated field tract has underlying multiplicative monoid $k$. The null set consists of the formal sums which evaluate to zero in $k$:
        \[
            N_k
            =
            \left\{
                \sum_{i=1}^m a_i \in \mathbb{N}[k^\times]
                \ \middle|\
                \sum_{i=1}^m a_i=0 \text{ in } k
            \right\}.
        \]
        Thus tracts generalize fields.

        \item \textbf{The Krasner hyperfield $\mathbb{K}$.} The Krasner hyperfield has underlying set $\{0,1\}$, with the usual multiplication. Its null set consists of the empty sum and all formal sums with at least two nonzero terms:
        \[
            N_{\mathbb{K}}
            =
            \{0\}
            \cup
            \{1+\cdots+1 \mid \text{there are at least two summands}\}.
        \]
        Intuitively, one should think of $1\in\mathbb{K}$ as representing an unspecified nonzero element of a field. The relations
        \[
            1+1\in N_{\mathbb{K}}
            \qquad\text{and}\qquad
            1+1+1\in N_{\mathbb{K}}
        \]
        reflect the fact that the sum of two nonzero field elements may be zero or nonzero.

        \item \textbf{The tropical hyperfield $\mathbb{T}$.} The tropical hyperfield has underlying set $\R\cup\{-\infty\}$. Multiplication is ordinary addition on $\R$, with identity $0$ and absorbing element $-\infty$. The null set consists of the empty sum together with the nonempty 
        formal sums for which the maximum of the $c_i$ is attained at least twice.
    \end{enumerate}
    In both the Krasner and tropical hyperfields, we have $1=-1$. This feature is important in the theory of polymatroid representations.
\end{example}

\begin{definition}
    Let $(F_1,N_{F_1})$ and $(F_2,N_{F_2})$ be tracts. A \emph{morphism of tracts} $\varphi\colon F_1\to F_2$ is a homomorphism of monoids with zero such that the induced map
    \[
        \varphi_* \colon \mathbb{N}[F_1^\times]\to \mathbb{N}[F_2^\times]
    \]
    sends $N_{F_1}$ into $N_{F_2}$. If $\varphi$ is bijective and its inverse is also a morphism of tracts, then $\varphi$ is an \emph{isomorphism}.
\end{definition}

For a more detailed discussion of tracts and the related notion of pastures, see \cite{baker2020foundationsmatroidsimatroids}. For polymatroid representations over tracts, see \cite{baker2025representationtheorypolymatroids}.

\subsection{Exponential Generating Functions with Coefficients}

We now introduce the objects which will be used for both Lorentzian polynomials and polymatroid representations.

Fix ordered variables $w=(w_1,\dots,w_n)$. For $\alpha\in\Zgeqn$, write
\[
    \alpha! = \prod_{i=1}^n \alpha_i!,
    \qquad
    w^\alpha = \prod_{i=1}^n w_i^{\alpha_i}.
\]

\begin{definition}
    For a commutative monoid with zero $F$, we write $\mathcal{E}_F(w)$ for the set of finite formal exponential generating functions
    \[
        f(w)=\sum_{\alpha\in\Zgeqn} c_\alpha \frac{w^\alpha}{\alpha!},
        \qquad c_\alpha\in F.
    \]
    Here ``finite'' means that all but finitely many coefficients $c_\alpha$ are zero, and the division by $\alpha!$ is purely formal. The \emph{support} of $f$ is
    \[
        \supp(f)=\{\alpha\in\Zgeqn \mid c_\alpha\neq 0\}.
    \]
    We write
    \[
        \mathcal{E}^{\mathrm{mc}}_F(w)
        =
        \{f\in\mathcal{E}_F(w) \mid \supp(f) \text{ is M-convex}\}.
    \]
\end{definition}

The normalization by $\alpha!$ is convenient because translation of exponent vectors agrees formally with repeated differentiation and integration, without introducing extra scalar factors.

\begin{definition} \label{def:projective_equivalence}
    Suppose that every nonzero element of $F$ is a unit. Two elements $f,g\in\mathcal{E}_F(w)$ are \emph{projectively equivalent} if there exists $\lambda\in F^\times$ such that $f=\lambda g$.
\end{definition}

We emphasize that the notion of projective equivalence in Definition~\ref{def:projective_equivalence} uses a global scalar relating the two generating functions.

Lorentzian polynomials naturally give elements of $\mathcal{E}^{\mathrm{mc}}_{\R_{\geq 0}}(w)$. Similarly, a polymatroid representation over a tract $F$ gives an element of $\mathcal{E}^{\mathrm{mc}}_F(w)$ whose coefficients satisfy the appropriate Grassmann-Pl\"ucker relations. Thus both settings can be treated as coefficient data supported on an M-convex set.

We now define deletion, contraction, and translation for elements of $\mathcal{E}^{\mathrm{mc}}_F(w)$. These operations are chosen so that their supports agree with the corresponding embedded minors of M-convex sets. To avoid confusing translation of polynomials with addition of polynomials, we denote translation by $T_\tau f$ rather than $f+\tau$.

\begin{definition}\label{def:embedded_minor_gen_poly}
    Let
    \[
        f(w)=\sum_{\alpha\in J} c_\alpha \frac{w^\alpha}{\alpha!}
    \]
    be an element of $\mathcal{E}^{\mathrm{mc}}_F(w)$ with support $J$. Let $\delminus$ and $\delplus$ denote the coordinatewise minimum and maximum vectors of $J$.

    \begin{enumerate}
        \item If $\tau\in\mathbb{Z}^n$ satisfies $\tau\geq -\delminus$, define the \emph{translation} of $f$ by $\tau$ to be
        \[
            T_\tau f
            \coloneqq
            \sum_{\alpha\in J}
            c_\alpha
            \frac{w^{\alpha+\tau}}{(\alpha+\tau)!}.
        \]

        \item If $\nu\in\Zgeqn$ is effectively coindependent in $J$, define the \emph{deletion} of $\nu$ from $f$ by
        \[
            f\setminus \nu
            \coloneqq
            \sum_{\alpha\in J\setminus \nu}
            c_\alpha
            \frac{w^\alpha}{\alpha!}.
        \]

        \item If $\mu\in\Zgeqn$ is effectively independent in $J$, define the \emph{contraction} of $\mu$ from $f$ by
        \[
            f/\mu
            \coloneqq
            \sum_{\alpha-\mu\in J/\mu}
            c_\alpha
            \frac{w^{\alpha-\mu}}{(\alpha-\mu)!}.
        \]
    \end{enumerate}
\end{definition}

Translation shifts the exponent vectors without changing the coefficients. Deletion and contraction remove terms: deletion keeps the terms whose exponents remain in $J\setminus\nu$, while contraction keeps the terms which survive the contraction and then shifts the support by $-\mu$. By construction,
\[
    \supp(T_\tau f)=J+\tau,
    \qquad
    \supp(f\setminus\nu)=J\setminus\nu,
    \qquad
    \supp(f/\mu)=J/\mu.
\]

\begin{remark}\label{remark:deletion_contraction_match_truncation}
    These operations can also be described using the truncation notation of \cite{ross2023duallylorentzianpolynomials}. If $f_{\leq u}$ denotes the polynomial obtained by keeping only terms with exponent at most $u$, and $f_{\geq \ell}$ denotes the polynomial obtained by keeping only terms with exponent at least $\ell$, then
    \[
        f\setminus\nu = f_{\leq \delplus - \nu},
        \qquad
        f/\mu = \partial^\mu\left(f_{\geq \delminus+\mu}\right).
    \]
    Here $\partial^\mu$ denotes the partial derivative with multiplicity vector $\mu$. The exponential normalization is what makes the contraction formula agree with this description.
\end{remark}

We now define $\kappa$-retractions of these objects. Let
\[
    f(w)=\sum_{\alpha\in J} c_\alpha \frac{w^\alpha}{\alpha!}
\]
be an element of $\mathcal{E}^{\mathrm{mc}}_F(w)$ with support $J$, and fix a cage $\kappa$. Since every exponent of $f$ lies in $J$, replacing $\kappa$ by its coordinatewise minimum with $\delplus$ does not change $\kproj(J)$ or $\RetK(J)$. Therefore, in the embedded-minor description below, we assume without loss of generality that $\kappa\leq\delplus$. 

The retraction of the support alone does not determine which coefficients should be kept. To specify the surviving coefficients, we choose a section of the projection map
\[
    \kproj\colon \Jmax \to \RetK(J).
\]

\begin{definition}\label{def:Ret_polynomial}
    Suppose that $\tau\in\Zgeqn$ is such that $\RetK(J)+\tau$ is a section of $\kproj$ contained in $\Jmax$. We define the \emph{$\tau$-section retraction} of $f$ to be
    \[
        \RetK^\tau(f)
        \coloneqq
        \sum_{x\in \RetK(J)}
        c_{x+\tau}
        \frac{w^x}{x!}.
    \]
\end{definition}

Thus, $\RetK^\tau(f)$ is obtained by restricting the coefficients of $f$ to the section $\RetK(J)+\tau\subseteq J$, and then translating that section back down to $\RetK(J)$. In particular,
\[
    \supp(\RetK^\tau(f))=\RetK(J).
\]

Corollary~\ref{corollary:description_of_embedded_minor} gives an explicit description of the $\tau$-section retraction. Fix $\alpha\in\Jmax$, and set
\[
    \tau=\alpha-\kproj(\alpha)
    \qquad\text{and}\qquad
    \nu=\delplus-\kappa-\tau.
\]
Since $\tau \leq \delplus - \kappa$, this gives a nonnegative deletion vector $\nu$. The corollary gives that
\[
    \RetK(J)=J\setminus\nu-\tau,
\]
or equivalently $J\setminus\nu=\RetK(J)+\tau$. Thus, on the level of supports, the $\tau$-section retraction is obtained by first deleting $\nu$ and then translating by $-\tau$. The same holds for $f$:
\[
    \RetK^\tau(f)=T_{-\tau}(f\setminus\nu).
\]
Indeed, if
\[
    f(w)=\sum_{\beta\in J} c_\beta \frac{w^\beta}{\beta!},
\]
then
\begin{align*}
    T_{-\tau}(f\setminus\nu)
    &=
    T_{-\tau}
    \left(
        \sum_{\beta\in J\setminus\nu}
        c_\beta \frac{w^\beta}{\beta!}
    \right) \\
    &=
    T_{-\tau}
    \left(
        \sum_{\beta\in \RetK(J)+\tau}
        c_\beta \frac{w^\beta}{\beta!}
    \right) \\
    &=
    \sum_{x\in\RetK(J)}
    c_{x+\tau}\frac{w^x}{x!}
    =
    \RetK^\tau(f).
\end{align*}

The notation includes $\tau$ because, for arbitrary coefficient data, different sections may produce different polynomials. In the next sections, we impose additional structure on the coefficients. For Lorentzian polynomials and representations over a tract, the relevant choices of section behave well; in particular, the resulting retractions can be compared up to projective equivalence as in Definition~\ref{def:projective_equivalence}.

\section{Lorentzian Polynomials and Caged Retractions}
\label{section:lorentzian_polynomials}

The goal of this section is to show that the Lorentzian property is preserved by caged retractions. We first observe that embedded minors preserve the Lorentzian property. We then prove that different choices of section give projectively equivalent Lorentzian polynomials.

We begin by recalling the definition of Lorentzian polynomials.

\begin{definition}\label{def:lorentzian_poly}
    The set of \emph{Lorentzian polynomials} of degree $r$ in $n$ variables is denoted by $\Lnr$. For $r=0$ or $r=1$, every nonzero homogeneous polynomial of degree
    $r$ with nonnegative coefficients is Lorentzian. For $r \geq 2$, a homogeneous polynomial
    \[
        f(w)\in\R_{\geq 0}[w_1,\dots,w_n]
    \]
    of degree $r$ lies in $\Lnr$ if both of the following conditions hold:
    \begin{enumerate}
        \item The support $\supp(f)$ is an M-convex set.
        \item For every $\beta\in\Delta_n^{r-2}$, the Hessian $\hess(\partial^\beta f)$ has at most one positive eigenvalue.
    \end{enumerate}
    Here $\hess$ denotes the Hessian matrix of second derivatives.
\end{definition}

We say that a polynomial has the Lorentzian property if it lies in $\Lnr$ for some $n$ and $r$. For an in-depth discussion of Lorentzian polynomials, including alternative characterizations, see \cite{branden2024lorentzianpolynomials}.

In the language of the previous section, a Lorentzian polynomial is an element of $\mathcal{E}^{\mathrm{mc}}_{\R_{\geq 0}}(w)$ whose coefficients satisfy the Hessian condition. Homogeneity follows from the support condition, since every M-convex set lies in $\dnr$ for some $r$.

Throughout this section, we write a Lorentzian polynomial with support $J$ in exponential form as
\[
    f(w)=\sum_{\alpha\in J} c_\alpha \frac{w^\alpha}{\alpha!},
    \qquad c_\alpha>0.
\]

\begin{theorem}\label{thm:embedded_minor_of_lorentzian_works}
    Let $f\in\Lnr$ with $\supp(f)=J$. Let $\mu,\nu\in\Zgeqn$ be effectively independent and effectively coindependent in $J$, respectively, and let $\tau\in\mathbb{Z}^n$ satisfy $\tau\geq-\delminus$. Then:
    \begin{enumerate}
        \item $f\setminus\nu$ is Lorentzian.
        \item $f/\mu$ is Lorentzian.
        \item $T_\tau f$ is Lorentzian.
    \end{enumerate}
\end{theorem}

\begin{proof}
    By Remark~\ref{remark:deletion_contraction_match_truncation}, deletion and contraction are described by truncation and differentiation. These operations preserve the Lorentzian property by Proposition~3.3 of \cite{ross2023duallylorentzianpolynomials}, together with the fact that differentiation preserves the Lorentzian property. Thus it remains to prove that translations preserve the Lorentzian property.

    Since any translation is a composition of translations by $\pm e_i$, it suffices to consider translation by a single coordinate vector. Translation by $-e_i$ is differentiation with respect to $w_i$, so it preserves the Lorentzian property. It remains to consider translation by $+e_i$. Without loss of generality, assume $i=1$.

    Let
    \[
        g=T_{e_1}f.
    \]
    Equivalently,
    \[
        g(w_1,\dots,w_n)
        =
        \int_0^{w_1} f(t,w_2,\dots,w_n)\,dt.
    \]
    Thus
    \[
        g(w)
        =
        \sum_{\alpha\in J}
        c_\alpha
        \frac{w^{\alpha+e_1}}{(\alpha+e_1)!}.
    \]
    The support of $g$ is $J+e_1$, which is M-convex, and $g$ is homogeneous of degree $r+1$. It remains to check the Hessian condition.

    If $r=0$, then $g$ has degree $1$ and is Lorentzian by definition.
    We may therefore assume that $r\geq 1$.

    Let $\beta\in\Delta_n^{r-1}$. We show that the Hessian of $\partial^\beta g$ has at most one positive eigenvalue. If $\beta_1\geq 1$, then
    \[
        \partial^\beta g
        =
        \partial^{\beta-e_1} f,
    \]
    so the Hessian condition follows from the fact that $f$ is Lorentzian.

    Now suppose $\beta_1=0$. If $\partial^\beta g$ is identically zero, there is nothing to prove. Otherwise, since every term of $g$ is divisible by $w_1$ and $\partial^\beta g$ has degree two, we can write
    \[
        \partial^\beta g
        =
        a w_1^2+\sum_{j=2}^n b_j w_1w_j
    \]
    for some $a,b_j\in\R_{\geq 0}$. Its Hessian has the form
\[
    \begin{pmatrix}
        2a & b^T \\
        b & 0
    \end{pmatrix},
\]
where $b=(b_2,\dots,b_n)^T$. By the Cauchy interlacing theorem, this matrix has at most one positive eigenvalue, as desired. It follows that $T_\tau f$ is Lorentzian for every allowable translation vector $\tau$.
\end{proof}

We have established that embedded minors preserve the Lorentzian property. Applying this to the embedded-minor description of caged retractions gives the following corollary.

\begin{corollary}\label{cor:section_retraction_lorentzian}
    Let $f\in\Lnr$ with $\supp(f)=J$, and let $\tau$ define a section of
    \[
        \kproj\colon \Jmax\to\RetK(J).
    \]
    Then $\RetK^\tau(f)$ is Lorentzian.
\end{corollary}

\begin{proof}
    By Corollary~\ref{corollary:description_of_embedded_minor}, the polynomial $\RetK^\tau(f)$ is obtained from $f$ by deletion followed by translation. The result follows from Theorem~\ref{thm:embedded_minor_of_lorentzian_works}.
\end{proof}

\begin{example}
    Let $J$ be any M-convex set, and let
    \[
        f(w)=\sum_{\alpha\in J}\frac{w^\alpha}{\alpha!}
    \]
    be its exponential generating polynomial. By Theorem~3.10 of \cite{branden2024lorentzianpolynomials}, this polynomial is Lorentzian. For every choice of section-defining vector $\tau$, the retraction $\RetK^\tau(f)$ is the exponential generating polynomial of $\RetK(J)$.
\end{example}

\begin{example}
    Consider the Lorentzian polynomial
    \[
        f(x,y)=\frac{6}{3!\cdot 0!}x^3+\frac{4}{2!\cdot 1!}x^2y+\frac{2}{1!\cdot 2!}xy^2.
    \]
    Let $J=\supp(f)$ and take the cage $\kappa=(1,1)$. Then
    \[
        \RetK(J)=\{(1,1)\}.
    \]
    There are two choices of section:
    \[
        \tau=(1,0)
        \qquad\text{and}\qquad
        \tau'=(0,1).
    \]
    The corresponding section retractions are
    \[
        \RetK^\tau(f)=4xy,
        \qquad
        \RetK^{\tau'}(f)=2xy.
    \]
    Thus different choices of section can produce different polynomials. However, in this example the two retractions are projectively equivalent. We now show that this occurs in general.
\end{example}

The following lemma is an immediate corollary of Theorem~1.11 of \cite{zhangSchurComplementIts2005}.

\begin{lemma}\label{lemma:1_pos_eigval_gives_rk1}
    Let $A\in\R^{m\times m}$ be a symmetric matrix and let $C\in\R^{m\times(n-m)}$. If the matrix
    \[
        M=
        \begin{pmatrix}
            A & C \\
            C^T & 0
        \end{pmatrix}
    \]
    has at most one positive eigenvalue, then $\mathrm{rank}(C)\leq 1$.
\end{lemma}

We will use this lemma to show the retractions coming from different choices of section are projectively equivalent. Recall from Lemma~\ref{lemma:collection_of_sections_is_M-convex} that the set of section-defining vectors $\tau$ is M-convex. Let us consider the case of two sections which differ by an elementary exchange.

\begin{theorem}\label{theorem:projective_eq_of_lorentzian_sections}
    Let $f\in\Lnr$ with $\supp(f)=J$. Let $\tau$ and $\tau'=\tau+e_k-e_l$ be two section-defining vectors for the projection
    \[
        \kproj\colon \Jmax\to\RetK(J).
    \]
    Then there exists a scalar $\lambda>0$ such that
    \[
        \RetK^\tau(f)=\lambda\RetK^{\tau'}(f).
    \]
\end{theorem}
\begin{proof}
    By definition of the $\tau$-section retraction,
    \[
        \RetK^\tau(f)
        =
        \sum_{x\in\RetK(J)}
        c_{x+\tau}\frac{w^x}{x!}
        \qquad\text{and}\qquad
        \RetK^{\tau'}(f)
        =
        \sum_{x\in\RetK(J)}
        c_{x+\tau'}\frac{w^x}{x!}.
    \]
    If $\RetK(J)$ consists of a single element, the result is immediate. Thus assume $|\RetK(J)|\geq 2$. By the connectivity of the basis-exchange graph for polymatroids, it suffices to prove that for any elementary exchange $y=x+e_i-e_j$ in $\RetK(J)$, we have
    \[
        \frac{c_{x+\tau'}}{c_{x+\tau}}
        =
        \frac{c_{y+\tau'}}{c_{y+\tau}}.
    \]

    By Lemma~\ref{lemma:section_vectors_saturate_coordinates}, the indices $k$ and $l$ cannot be equal to either $i$ or $j$. We use this in the computations below.

    By Theorem~\ref{theorem:structure_theorem_for_mconvex}, there is an element $\alpha\in J$ lying in the fiber $\kproj^{-1}(x)$ such that $\alpha=x+\tau$. Similarly, the element in this fiber corresponding to $\tau'$ is $\alpha'=x+\tau'$. Since $\tau'=\tau+e_k-e_l$, we have $\alpha'=\alpha+e_k-e_l$. Choose $\beta=\alpha-e_l-e_j$. Since the indices $i,j,k,l$ are distinct, and since $\alpha_l>\kappa_l$ and $\alpha_j=x_j>0$, we have $\beta\in\Zgeqn$. Also $\beta\in\Delta_n^{r-2}$. Rewriting the desired equality in terms of $\beta$, it suffices to show
    \[
        \frac{c_{\beta+e_k+e_j}}{c_{\beta+e_l+e_j}}
        =
        \frac{c_{\beta+e_k+e_i}}{c_{\beta+e_l+e_i}}.
    \]

    Since $\alpha'=\alpha+e_k-e_l$ has the same $\kappa$-projection as $\alpha$, we must have $\alpha_k\geq\kappa_k$ and $\alpha_l>\kappa_l$. Hence $\beta_k\geq\kappa_k$ and $\beta_l\geq\kappa_l$. On the other hand, since $y=x+e_i-e_j\in\RetK(J)$, Lemma~\ref{lemma:fiber_exchange_like_retraction} implies that $\alpha_i=x_i$ and $\alpha_j=x_j$. Since $y_i=x_i+1\leq\kappa_i$, we have $\beta_i=\alpha_i=x_i<\kappa_i$; and since $x_j\leq\kappa_j$, we have $\beta_j=\alpha_j-1=x_j-1<\kappa_j$.

    Reordering the variables, we may assume that the indices $\{1,\dots,m\}$ satisfy $\beta_s\geq\kappa_s$, while the indices $\{m+1,\dots,n\}$ satisfy $\beta_s<\kappa_s$. Under this ordering, $k,l\in\{1,\dots,m\}$ and $i,j\in\{m+1,\dots,n\}$.

    We compute
    \[
        \partial^\beta f
        =
        \sum_{\substack{u\leq v\\ \beta+e_u+e_v\in J}}
        c_{\beta+e_u+e_v}
        \frac{w_uw_v}{(e_u+e_v)!}.
    \]
    Since $\partial^\beta f$ is homogeneous of degree two and written in exponential form, its Hessian matrix $H$ has entries $H_{uv}=c_{\beta+e_u+e_v}$. We write $H$ in block form
    \[
        H=
        \begin{pmatrix}
            A & C \\
            C^T & B
        \end{pmatrix},
    \]
    where $A$ is an $m\times m$ matrix and $B$ is an $(n-m)\times(n-m)$ matrix.

    We claim that the off-diagonal entries of $B$ are zero. Let $u$ and $v$ be distinct indices in the block $B$. Since $\beta_u<\kappa_u$ and $\beta_v<\kappa_v$, adding $e_u+e_v$ increases the size of the $\kappa$-projection by two. Meanwhile, passing from $\alpha$ to $\beta=\alpha-e_l-e_j$ decreases the size of the $\kappa$-projection by only one as $\alpha_l > \kappa_l$. Hence $|\kproj(\beta+e_u+e_v)|=|\kproj(\alpha)|+1$, so by Corollary~\ref{corollary:Jmax_size_description}, the vector $\beta+e_u+e_v$ cannot lie in $J$. Therefore the off-diagonal entries of $B$ are zero. Since the coefficients of $f$ are nonnegative, $B$ is a nonnegative diagonal matrix, and hence is positive semidefinite.

    Since $f$ is Lorentzian, $H$ has at most one positive eigenvalue. By Weyl's monotonicity theorem, replacing the positive semidefinite block $B$ with the zero matrix preserves this property. Thus the matrix
    \[
        \begin{pmatrix}
            A & C \\
            C^T & 0
        \end{pmatrix}
    \]
    has at most one positive eigenvalue.

    The matrix $C$ is not zero. For instance, $C_{l,j}=c_{\beta+e_l+e_j}=c_\alpha\neq 0$. Thus, by Lemma~\ref{lemma:1_pos_eigval_gives_rk1}, the matrix $C$ has rank exactly one.

    The entries appearing below are nonzero because they are the coefficients of $x+\tau$, $x+\tau'$, $y+\tau$, and $y+\tau'$. Since $C$ has rank one, its columns are proportional. In particular, applying this proportionality to the columns indexed by $i$ and $j$, and to the rows indexed by $k$ and $l$, gives
    \[
        \frac{C_{k,i}}{C_{l,i}}
        =
        \frac{C_{k,j}}{C_{l,j}}.
    \]
    Substituting $C_{u,v}=c_{\beta+e_u+e_v}$ gives
    \[
        \frac{c_{\beta+e_k+e_i}}{c_{\beta+e_l+e_i}}
        =
        \frac{c_{\beta+e_k+e_j}}{c_{\beta+e_l+e_j}},
    \]
    which is the desired equality.
\end{proof}

\begin{corollary}\label{cor:all_lorentzian_sections_projectively_equiv}
    Let $f\in\Lnr$ with $\supp(f)=J$. If $\tau$ and $\tau'$ are any two section-defining vectors for
    \[
        \kproj\colon \Jmax\to\RetK(J),
    \]
    then $\RetK^\tau(f)$ and $\RetK^{\tau'}(f)$ are projectively equivalent.
\end{corollary}

\begin{proof}
    By Lemma~\ref{lemma:collection_of_sections_is_M-convex}, the set of section-defining vectors is M-convex. Hence its basis-exchange graph is connected. The result follows by applying Theorem~\ref{theorem:projective_eq_of_lorentzian_sections} along a path between $\tau$ and $\tau'$.
\end{proof}

Since the different section retractions are projectively equivalent, the retraction of a Lorentzian polynomial is well-defined up to projective equivalence. Because $\R_{\geq 0}$ has addition, one can also choose a canonical representative by summing over all sections.

\begin{definition}\label{def:lorentzian_retraction_sum}
    Let $f\in\Lnr$ with support $J$. We define
    \[
        \RetK(f)
        \coloneqq
        \sum_\tau \RetK^\tau(f),
    \]
    where $\tau$ ranges over all section vectors for $\kproj\colon\Jmax\to\RetK(J)$.
\end{definition}

Equivalently, using the description of section vectors, we have
\[
    \RetK(f)
    =
    \sum_{\alpha\in\Jmax}
    c_\alpha
    \frac{w^{\kproj(\alpha)}}{\kproj(\alpha)!}.
\]
This expression is the polynomial analogue of the definition of $\RetK(J)$ as the set of maximal elements of $\kproj(J)$.

\begin{corollary}\label{cor:lorentzian_retraction_sum}
    Let $f\in\Lnr$ with $\supp(f)=J$. Then $\RetK(f)$ is Lorentzian.
\end{corollary}

\begin{proof}
    By Corollary~\ref{cor:all_lorentzian_sections_projectively_equiv}, the section retractions $\RetK^\tau(f)$ are pairwise projectively equivalent. Therefore there exists a fixed section vector $\tau_0$ and positive scalars $\lambda_\tau$ such that
    \[
        \RetK^\tau(f)=\lambda_\tau\RetK^{\tau_0}(f)
    \]
    for every section vector $\tau$. Hence
    \[
        \RetK(f)
        =
        \left(\sum_\tau\lambda_\tau\right)
        \RetK^{\tau_0}(f).
    \]
    Since $\RetK^{\tau_0}(f)$ is Lorentzian by Corollary~\ref{cor:section_retraction_lorentzian}, it follows that $\RetK(f)$ is Lorentzian.
\end{proof}

\begin{remark}
    When $\kappa=\ones$, the retraction $\RetK(J)$ is the matroid retraction. In this case, if the multi-affine part of $J$ is nonempty, then $\RetK(J)$ is precisely the support of the multi-affine part. Thus the preceding results recover the fact that the multi-affine part of a Lorentzian polynomial is Lorentzian.
\end{remark}

Because the product of two Lorentzian polynomials is Lorentzian, see \cite[Corollary~2.32]{branden2024lorentzianpolynomials}, we obtain a Lorentzian version of caged polymatroid union.

\begin{corollary}[Lorentzian Caged Union]\label{cor:lorentzian_caged_union}
    Let $f\in L_n^{r_1}$ and $g\in L_n^{r_2}$ be Lorentzian polynomials with supports $J_1$ and $J_2$. Then every section retraction of $fg$ is Lorentzian and is supported on
    \[
        \RetK(J_1+J_2).
    \]
    In particular, if $J_1$ and $J_2$ are $\kappa$-caged, then every section retraction of $fg$ is supported on the $\kappa$-polymatroid union
    \[
        J_1\vee J_2=\RetK(J_1+J_2).
    \]
\end{corollary}

\begin{proof}
    By \cite[Corollary~2.32]{branden2024lorentzianpolynomials}, the product $fg$ is Lorentzian. Its support is the Minkowski sum $J_1+J_2$. Therefore Corollary~\ref{cor:section_retraction_lorentzian} implies that every section retraction of $fg$ is Lorentzian and supported on $\RetK(J_1+J_2)$.
\end{proof}

\section{Representations and Caged Retractions}
\label{section:representations}

The goal of this section is to show that caged retractions behave well for representations of polymatroids over tracts. We recall strong Grassmann-Plücker functions over near-idempotent tracts, following \cite{baker2025representationtheorypolymatroids}, and encode them as generating functions. We then use embedded minors to obtain representations of $\RetK(J)$ and prove that different choices of section give projectively equivalent representations.

\subsection{Representations as Generating Functions}

We restrict to the class of near-idempotent tracts.

\begin{definition}
    A tract $(F,N_F)$ is \emph{near-idempotent} if $1=-1$ and there exists some $x\in F^\times$ such that
    \[
        1+1+x\in N_F.
    \]
\end{definition}

Both the Krasner hyperfield and the tropical hyperfield are near-idempotent. It turns out that a polymatroid which is not a translate of a matroid has no representations over a tract which is not near-idempotent; see \cite{baker2025representationtheorypolymatroids} for details. For this reason, we restrict to near-idempotent tracts.

The formulation in \cite{baker2025representationtheorypolymatroids} primarily defines representations after reducing the polymatroid, so that $\delminus=0$. Equivalently, they also give a formulation in which one does not reduce first. We use this latter formulation, since it matches the generating-function operations defined above.

\begin{definition}\label{definition:strong_GP_representation}
    Let $(F,N_F)$ be a near-idempotent tract, and let $J\subseteq\dnr$ be an M-convex set. A \emph{strong Grassmann-Pl\"ucker representation} of $J$ over $F$ is a map
    \[
        \rho\colon\dnr\to F
    \]
    satisfying the following conditions:
    \begin{enumerate}
        \item The support of $\rho$ is $J$. That is, for any $\alpha\in\dnr$,
        \[
            \rho(\alpha)\in F^\times
            \quad\Longleftrightarrow\quad
            \alpha\in J.
        \]

                \item The coefficients satisfy the Grassmann-Pl\"ucker relations. Namely, for every $2\leq s\leq r$, every $\alpha\in\Delta_n^{r-s}$, and every choice of indices
        \[
            i_0,\dots,i_s,j_2,\dots,j_s\in[n]
        \]
        such that
        \[
            \delminus\leq\alpha
            \qquad\text{and}\qquad
            \alpha+e_{i_0}+\cdots+e_{i_s}+e_{j_2}+\cdots+e_{j_s}\leq\delplus,
        \]
        we have
        \[
            \sum_{k=0}^s
            \rho(\alpha+e_{i_0}+\cdots+\widehat{e_{i_k}}+\cdots+e_{i_s})
            \rho(\alpha+e_{i_k}+e_{j_2}+\cdots+e_{j_s})
            \in N_F.
        \]
        Here the hat means that the corresponding term is omitted.
    \end{enumerate}
\end{definition}

The most basic example is the representation of an M-convex set over the Krasner hyperfield $\mathbb{K}$, where $\rho$ is the indicator function of $J$. In this case, the Grassmann-Pl\"ucker relations are equivalent to the base exchange axiom. Representations over the tropical hyperfield $\mathbb{T}$ are closely related to M-convex functions in the sense of Murota \cite{murotaDiscreteConvexAnalysis2003}. For matroids, these recover valuated matroids.

Given a representation $\rho$ of $J$ over a tract $F$, we encode it by the formal exponential generating function
\[
    f_\rho(w)
    \coloneqq
    \sum_{\alpha\in J}
    \rho(\alpha)\frac{w^\alpha}{\alpha!}.
\]
Thus $f_\rho\in\mathcal{E}^{\mathrm{mc}}_F(w)$ and $\supp(f_\rho)=J$. Conversely, an element $f\in\mathcal{E}^{\mathrm{mc}}_F(w)$ whose coefficients satisfy the Grassmann-Pl\"ucker relations defines a representation of its support over $F$. We will therefore identify representations with their generating functions.

This viewpoint allows us to apply the embedded minor operations from Definition~\ref{def:embedded_minor_gen_poly} directly to representations. 

\begin{theorem}\label{theorem:embedded_minors_preserve_representations}
    Let $f\in\mathcal{E}^{\mathrm{mc}}_F(w)$ be a representation of an M-convex set $J$ over a near-idempotent tract $F$. Let $\mu,\nu\in\Zgeqn$ be effectively independent and effectively coindependent in $J$, respectively, and let $\tau\in\mathbb{Z}^n$ satisfy $\tau\geq-\delminus$. Then:
    \begin{enumerate}
        \item $f\setminus\nu$ is a representation of $J\setminus\nu$ over $F$.
        \item $f/\mu$ is a representation of $J/\mu$ over $F$.
        \item $T_\tau f$ is a representation of $J+\tau$ over $F$.
    \end{enumerate}
    Moreover, these operations are natural and induce the corresponding maps on foundations.
\end{theorem}

\begin{proof}
    This is Theorem~7.1 of \cite{baker2025representationtheorypolymatroids}, translated into the generating-function notation of Definition~\ref{def:embedded_minor_gen_poly}.
\end{proof}

Applying Theorem~\ref{theorem:embedded_minors_preserve_representations} to the embedded-minor description of caged retractions gives the following.

\begin{corollary}\label{corollary:section_retraction_representation}
    Let $f\in\mathcal{E}^{\mathrm{mc}}_F(w)$ be a representation of an M-convex set $J$ over a near-idempotent tract $F$. Let $\tau$ be a section-defining vector for
    \[
        \kproj\colon\Jmax\to\RetK(J).
    \]
    Then $\RetK^\tau(f)$ is a representation of $\RetK(J)$ over $F$.
\end{corollary}

\begin{proof}
    By Corollary~\ref{corollary:description_of_embedded_minor}, the section retraction $\RetK^\tau(f)$ is obtained from $f$ by deletion followed by translation. The result follows from Theorem~\ref{theorem:embedded_minors_preserve_representations}.
\end{proof}

We now show that the representation of $\RetK(J)$ is independent of the choice of section up to projective equivalence as in Definition~\ref{def:projective_equivalence}. As in the Lorentzian case, it suffices to compare section-defining vectors differing by an elementary exchange.

\begin{theorem}\label{theorem:representations_projective_equivalence}
    Let $f\in\mathcal{E}^{\mathrm{mc}}_F(w)$ be a representation of an M-convex set $J$ over a near-idempotent tract $F$. Let $\tau$ and $\tau'=\tau+e_k-e_l$ be two section-defining vectors for
    \[
        \kproj\colon\Jmax\to\RetK(J).
    \]
    Then there exists $\lambda\in F^\times$ such that
    \[
        \RetK^\tau(f)=\lambda\RetK^{\tau'}(f).
    \]
\end{theorem}

\begin{proof}
    Write the coefficients of $f$ as $\rho$. By definition,
    \[
        \RetK^\tau(f)
        =
        \sum_{x\in\RetK(J)}
        \rho(x+\tau)\frac{w^x}{x!}
        \qquad\text{and}\qquad
        \RetK^{\tau'}(f)
        =
        \sum_{x\in\RetK(J)}
        \rho(x+\tau')\frac{w^x}{x!}.
    \]
    If $\RetK(J)$ consists of a single element, the result is immediate. Thus assume $|\RetK(J)|\geq 2$. By the connectivity of the basis-exchange graph for polymatroids, it suffices to show that for any elementary exchange $y=x+e_i-e_j$ in $\RetK(J)$, we have
    \[
        \frac{\rho(x+\tau')}{\rho(x+\tau)}
        =
        \frac{\rho(y+\tau')}{\rho(y+\tau)}.
    \]

    By Lemma~\ref{lemma:section_vectors_saturate_coordinates}, the indices $k$ and $l$ cannot be equal to either $i$ or $j$. Let $\alpha=x+\tau$ and set $\beta=\alpha-e_l-e_j$. As in the proof of Theorem~\ref{theorem:projective_eq_of_lorentzian_sections}, we have $\beta\in\Zgeqn$ and $\beta\in\Delta_n^{r-2}$. Moreover, the parameters $\beta,i,j,k,l$ satisfy the bounds required for the three-term Grassmann-Pl\"ucker relation. The lower bound follows from $x+\tau'=\alpha+e_k-e_l\in J$ and $y+\tau=\alpha+e_i-e_j\in J$. The upper bound follows coordinatewise from the fact that $x+\tau$, $x+\tau'$, $y+\tau$, and $y+\tau'$ all lie in $J$.

    Consider the three-term Grassmann-Pl\"ucker relation
    \[
        \rho(\beta+e_i+e_j)\rho(\beta+e_k+e_l)
        +
        \rho(\beta+e_i+e_k)\rho(\beta+e_j+e_l)
        +
        \rho(\beta+e_j+e_k)\rho(\beta+e_i+e_l)
        \in N_F.
    \]
    We claim that $\beta+e_i+e_j\notin J$. Indeed, the same size argument used in the Lorentzian proof shows that
    \[
        |\kproj(\beta+e_i+e_j)|
        =
        |\kproj(\alpha)|+1.
    \]
    By Corollary~\ref{corollary:Jmax_size_description}, this vector cannot lie in $J$. Hence the first product in the Grassmann-Pl\"ucker relation is zero.

    Therefore the relation reduces to the two-term relation
    \[
        \rho(\beta+e_i+e_k)\rho(\beta+e_j+e_l)
        +
        \rho(\beta+e_j+e_k)\rho(\beta+e_i+e_l)
        \in N_F.
    \]
    Since $F$ is near-idempotent, we have $1=-1$, and so a two-term null relation gives equality of the two terms. Thus
    \[
        \rho(\beta+e_i+e_k)\rho(\beta+e_j+e_l)
        =
        \rho(\beta+e_j+e_k)\rho(\beta+e_i+e_l).
    \]
    Substituting
    \[
    \begin{aligned}
        \beta+e_j+e_l &= x+\tau, &
        \beta+e_j+e_k &= x+\tau',\\
        \beta+e_i+e_l &= y+\tau, &
        \beta+e_i+e_k &= y+\tau'.
    \end{aligned}
    \]
    gives
    \[
        \rho(y+\tau')\rho(x+\tau)
        =
        \rho(x+\tau')\rho(y+\tau).
    \]
    Since all four coefficients are nonzero, this is equivalent to the desired ratio.
\end{proof}

\begin{corollary}\label{corollary:all_representation_sections_projectively_equiv}
    Let $f\in\mathcal{E}^{\mathrm{mc}}_F(w)$ be a representation of an M-convex set $J$ over a near-idempotent tract $F$. If $\tau$ and $\tau'$ are any two section-defining vectors for
    \[
        \kproj\colon\Jmax\to\RetK(J),
    \]
    then $\RetK^\tau(f)$ and $\RetK^{\tau'}(f)$ are projectively equivalent.
\end{corollary}

\begin{proof}
    By Lemma~\ref{lemma:collection_of_sections_is_M-convex}, the set of section-defining vectors is M-convex. Hence its basis-exchange graph is connected. Applying Theorem~\ref{theorem:representations_projective_equivalence} along a path from $\tau$ to $\tau'$ gives the result.
\end{proof}

The preceding corollary shows that, for representations over near-idempotent tracts, the caged retraction is well-defined up to projective equivalence. Equivalently, choosing a section of $\kproj\colon\Jmax\to\RetK(J)$ gives a representation of $\RetK(J)$, and different choices give projectively equivalent representations.

\begin{remark}
    Since $\RetK^\tau(f)$ is obtained from $f$ by an embedded minor operation, the construction is natural in the sense of \cite{baker2025representationtheorypolymatroids}. In particular, it induces the corresponding map on foundations. The projective-equivalence result above says in addition that the induced representation of $\RetK(J)$ is independent of the choice of section up to projective equivalence.
\end{remark}

\printbibliography
\end{document}